\documentclass[11pt,a4paper]{article}

\usepackage[math]{cellspace}
\usepackage[normalem]{ulem}
\usepackage{empheq}
\usepackage{psfrag}
\usepackage{amsmath,amssymb,cite,algorithm}
\usepackage{latexsym}
\usepackage{graphicx}
\usepackage{lscape}
\usepackage{afterpage}
\usepackage[color=green!40]{todonotes}
\usepackage{caption}
\usepackage{subcaption}
\usepackage{multirow}

\usepackage{hyperref}
\hypersetup{
	colorlinks=true,
	linkcolor=blue,
	filecolor=magenta,      
	urlcolor=cyan,
	citecolor=red,
}
\providecommand{\Rr}{\operatorname{R}}
\DeclareMathOperator*{\argmin}{argmin}
\DeclareMathOperator*{\rank}{rank}

\DeclareMathOperator*{\sgn}{sgn}

\newcommand{\Id}{\ensuremath{\operatorname{Id}}}

\newcommand{\ds}{\displaystyle}
\newcommand{\nexto}{\kern -0.54em}
\newcommand{\dR}{{\rm {I\ \nexto R}}}

\newcommand{\dZ}{{\cal Z \kern -0.7em Z}}
\newcommand{\dC}{{\rm\hbox{C \kern-0.8em\raise0.2ex\hbox{\vrule
				height5.4pt width0.7pt}}}}
\newcommand{\dQ}{{\rm\hbox{Q \kern-0.85em\raise0.25ex\hbox{\vrule
				height5.4pt width0.7pt}}}}
\newcommand{\proofbox}{\hspace{\fill}{$\Box$}}

\newtheorem{theorem}{Theorem}
\newtheorem{corollary}{Corollary}

\newtheorem{fact}{Fact}
\newtheorem{remark}{Remark}

\newenvironment{proof}{Proof.}{\proofbox}

\newcommand{\dom}{\ensuremath{\operatorname{dom}}}

\newcommand{\prox}{\ensuremath{\operatorname{Prox}}}

\newcommand{\scal}[2]{\left\langle{#1},{#2}  \right\rangle}

\definecolor{myblue}{rgb}{0.9,0.9,0.98}

\begin{document}

\author{Regina S. Burachik\thanks{School of Mathematical Sciences, Adelaide University, Australia. \\
Email: \texttt{regina.burachik@adelaide.edu.au}, \texttt{yalcin.kaya@adelaide.edu.au}.} \quad\ Bethany I. Caldwell\thanks{School of Mathematics and Physics, University of Queensland, Australia. \\ Email: \texttt{bethany.caldwell@uq.edu.au}}\quad\ C. Yal{\c c}{\i}n Kaya\footnotemark[1] \quad\ Walaa M. Moursi\thanks{Department of Combinatorics and Optimization, 
University of Waterloo,
Waterloo, Ontario N2L~3G1, Canada. Email: \texttt{walaa.moursi@uwaterloo.ca}.}
}

\title{\bf Best Approximation Optimal Control for Infeasible Double Integrator and Douglas--Rachford Algorithm}

\maketitle

\vspace*{-7mm}

\begin{abstract} 
{\noindent\sf
We consider the problem of finding (in some sense) the best approximation control for an infeasible double integrator.  The control function is constrained by upper and lower bounds that are too tight and thus cause infeasibility.  The infeasibility is characterized by a gap function (representing the separation between two constraint sets) whose squared ${\cal L}^2$-norm is to be minimized to find the best approximation control solution.  First, we review the existing results for problems involving a general linear control system. Then, for the infeasible double integrator problem, we present an analytical solution for the bang--bang control with at most one switching.  The infinite-dimensional optimization problem is reduced to the problem of solving two algebraic equations in two variables, to compute the switching time and gap function.  We discuss numerical approaches to solving the system of equations.  Finally, we describe the (relaxed) Douglas--Rachford algorithm for the double integrator problem and carry out numerical experiments to illustrate the implementation of the algorithm and test performance.}
\end{abstract}

\begin{verse}
\end{verse}

\begin{verse} 
\end{verse}

\pagestyle{myheadings}
\markboth{}{\sf\scriptsize Best Approximation Optimal Control and Douglas--Rachford Algorithm\ \ by Burachik, Caldwell, Kaya, and Moursi}

\section{Introduction}

The simple differential equation $\ddot{y}(t) = u(t)$ is called the {\em double integrator} since the solution $y(t)$ is obtained by integrating $u(t)$ twice.  For a mechanical system, $y(t)$, $\dot{y}(t)$, and $\ddot{y}(t)$ denote, respectively, the position, velocity, and acceleration of a unit point mass moving along a straight line on a frictionless plane, and $u(t)$ is the force (or control) acting on the point mass, all at a time instant $t$.  The double integrator also models the analogous rotational-mechanical and electrical systems~\cite{Wellstead2000}.

Optimal control of the double integrator, which is concerned with finding a control function~$u$ that minimizes a functional, such as ``energy'' or duration of time, is simple and yet rich enough to be studied when introducing and illustrating many basic and new concepts or when testing new numerical approaches in optimal control; see, for example,~\cite{BauBurKay2019, BurCalKayMou2024, Locatelli2017, BurKayLiu2023, BurKayLiu2026, Kaya2020}.  In this sense, the role of the double integrator is similar to that of the Dahlquist test problem, which is widely used for the stability and error analysis of ODE solvers~\cite{CorKayMoi2019, CorKay2026}.

If there is no constraint on the control $u(t)$, a closed-form analytical solution can be easily written for the {\em minimum-energy control} of the double integrator. Specifically speaking, the solution for the optimal control~$u$ minimizing the squared ${\cal L}^2$-norm of $u$, subject to the double integrator, can be expressed as an affine function in $t$ with coefficients defined by initial and terminal positions and velocities~\cite{BauBurKay2019}.  This constitutes the building block for cubic splines.  However, if the control $u(t)$ is constrained (typically by means of upper and lower bounds), then a solution (although known to be a concatenation of constant and affine functions) can only be found numerically.

The Douglas--Rachford (DR) algorithm was first used to find a numerical solution to the control-constrained minimum-energy control of the double integrator in~\cite{BauBurKay2019}, and it was shown that the DR algorithm was significantly more efficient (in terms of CPU time and accuracy) than the {\em first-discretize-then-optimize approach}, in which the optimal control problem is discretized and the resulting large-scale problem is solved by standard optimization software.  Later, in~\cite{BurCalKayMou2024}, an extension of the DR algorithm with a ``dual step'' was also used to solve the dual of the optimal control problem (as well as the primal problem) involving the double integrator, so that the necessary conditions of optimality for the control can be numerically verified.

In real life situations, the control $u(t)$, which typically represents the resources available to practitioners to achieve an optimal outcome, might be overly restricted, making the lower or upper bound (or both) of the control function too tight, thus resulting in an {\em infeasible} dynamical system and, in turn, an {\em infeasible} or {\em inconsistent} optimal control problem.  Such infeasible optimal control problems involving general linear ODEs, including the double integrator as a special case, have previously been studied in~\cite{BurKayMou2024}.  However, only the first-discretize-then-optimize approach was used in~\cite{BurKayMou2024} to obtain a {\em best approximation solution} (such that a measure of infeasibility is minimized).  In~\cite{BurKayMou2024}, the AMPL--Ipopt optimization suite~\cite{AMPL, WacBie2006} was implemented to solve large-scale optimization problems arising from the time-discretization of various example optimal control problems.  One of the novelties of the current paper is the study of the DR algorithm applied to an infeasible optimal control problem in Section~\ref{sec:DR}.  A preliminary application of the DR algorithm to infeasible optimal control problems can be found in~\cite[Section~6.7]{Caldwell2024}{, where the minimum-energy control of the double integrator and the damped-harmonic oscillator were considered.  In~\cite[Section~6.7]{Caldwell2024} performance comparisons amongst the relaxed DR algorithm, the Peaceman--Rachford algorithm and the AMPL--Ipopt suite are also included}. In the current paper, we study the infeasible optimal control of the double integrator in more detail.

In~\cite{BurKayMou2024}, the infeasible optimal control of the double integrator with $|u(t)| \le a$, $a$ a positive constant and specified initial and terminal states is considered where the infeasibility of the problem is caused by $a$ being too small.  In~\cite[Theorem~5]{BurKayMou2024}, a complete analytical {\em best approximation solution} to this problem is presented for the {\em critically feasible} case with $a = a_c$ such that the problem is feasible for $a = a_c$ but infeasible for any $a < a_c$.  Another novelty of the present paper is the complete analytical best approximation solution to the infeasible optimal control of the double integrator with $a < a_c$, as described in~Theorem~\ref{DI_infeas_control}.  Incidentally, the proof of Theorem~\ref{DI_infeas_control} turns out to be more challenging than that of \cite[Theorem~5]{BurKayMou2024}.

In~\cite{BurKayMou2024}, the constraint set is split into two, namely the {\em affine set} defined by the general linear dynamics with the specified boundary conditions and the {\em box} defined by the control constrained by simple bounds.  When the intersection of these two sets is empty (notably because of the tight bounds on the control), the optimal control problem becomes infeasible.  The {\em gap} between the two constraint sets is defined as the function that minimizes the ${\cal L}^2$-distance between these sets.  Then \cite[Theorem~3]{BurKayMou2024} states that the best approximation control that belongs to the box is of {\em bang--bang} type (i.e., the control function switches between the upper and lower bounds) with an unknown number of switchings and provides an expression for the gap function.

Theorem~\ref{DI_infeas_control} presented in the present paper asserts that the best approximation control is of bang--bang type with at most one switching, and the switching time and the gap function can be found by solving a system of two polynomial equations in two variables.  In other words, the switching time can be obtained with very high accuracy; so can the (affine) gap function.  We also prove an asymptotic result for the switching time as $a$ tends to zero, which might have practical value when $a$ is very small.  We consider a case study and investigate the solvability (in a reasonable number of iterations) of this system of equations by constructing the colour-coded number of iterations required to reach a solution of these equations when the Newton method and a generalized Newton method are used.

In Section~\ref{sec:DR}, we consider the same case study and compute the errors for the DR algorithm using the accurate switching times found as described above.  Our rationale behind the formulation of the DR algorithm for the double integrator problem here is to use it in the future for problems other than those involving the double integrator, for which an analytical solution cannot be obtained.

The paper is organized as follows.  In Section~\ref{sec:review}, we provide a review of the setting and results in~\cite{BurKayMou2024}.  In Section~\ref{sec:DI}, we describe the optimal control problem involving the double integrator and present the results on the bang--bang optimal control, the switching time and the gap function, and carry out a case study.  We describe the DR algorithm in Section~\ref{sec:DR}, and present numerical experiments involving the same case study as in Section~\ref{sec:DI}.  Finally, we provide concluding remarks in Section~\ref{sec:conclusion}.

\section{Review of the Case With General Linear Dynamics}
\label{sec:review}

In this section, we provide a review of the setting and the results in~\cite{BurKayMou2024} for the case where the dynamics are given by linear state differential equations in $x$, and $u$ appears linearly. Obviously, here the double integrator is a special case.

\subsection{The setting}

Consider an optimal control problem with the aim of finding a control $u$, which minimizes the functional
\begin{equation}  \label{obj_fun}
\ds \int_{0}^{1} f_0(x(t),u(t),t)\, dt\,,
\end{equation}
subject to the differential equation constraints
\begin{equation}  \label{ODE}
	\dot{x}(t) = A(t)\,x(t) + B(t)\,u(t)\,,\ \ \mbox{for a.e.\ } t\in[0,1]\,,
	\end{equation}
with $\dot{x} := dx/dt$, and the boundary conditions
\begin{equation}  \label{BC}
	\varphi(x(0),x(1)) = 0\,.
\end{equation}
The time horizon is set to $[0,1]$, but without loss of generality, it can be taken as any interval $[t_0,t_f]$, with $t_0$ and $t_f$ specified.  The integrand $f_0:\dR^n\times\dR^m\times[0,1]\to\dR_+$ is continuous.  Define the {\em state variable vector} $x:[0,1]\to\dR^n$ with $x(t) := (x_1(t)\,\ldots,x_n(t))\in\dR^n$ and the {\em control variable vector} $u:[0,1]\to\dR^m$ with $u(t) := (u_1(t)\,\ldots,u_m(t))\in\dR^m$. 
The time-varying matrices $A:[0,1]\to \dR^{n\times n}$ and $B:[0,1]\to\dR^{n\times m}$ are continuous.  The vector function $\varphi:\dR^{2n}\to\dR^r$, with $\varphi(x(0),x(1)) := (\varphi_1(x(0),x(1)),\ldots,\varphi_r(x(0),x(1)))\in\dR^r$, is affine.

In practical situations, it is realistic to restrict the control function $u$, for example, to impose simple bounds on the components of $u(t)$; namely,
\begin{equation}  \label{bounds}
	\underline{a}_i(t) \le u_i(t) \le \overline{a}_i(t)\,,\ \ \mbox{for a.e.\ } t\in[0,1]\,,
\end{equation}
where $\underline{a}_i,\overline{a}_i:[0,1]\to\dR$ are continuous, with $\underline{a}_i(t) \le \overline{a}_i(t)$, for all $t\in[0,1]$, $i = 1,\ldots,m$.

The objective functional in~\eqref{obj_fun} and the constraints in~\eqref{ODE}--\eqref{BC} and \eqref{bounds} can be put together to pose the  optimal control problem in a concise form as follows.
\[
\mbox{(P) }\left\{\begin{array}{rl}
	\ds\min_{u(\cdot)} & \ \ \ds\int_0^1 f_0(x(t),u(t),t)\, dt \\[4mm] 
	\mbox{subject to} & \ \ \dot{x}(t) = A(t)\,x(t) + B(t)\,u(t)\,,\ \ \mbox{for a.e.\ } t\in[0,1]\,, \\[2mm]
	& \ \ \varphi(x(0),x(1)) = 0\,, \\[2mm]
	& \ \ \underline{a}_i(t) \le u_i(t) \le \overline{a}_i(t)\,,\ \ \mbox{for a.e.\ } t\in[0,1]\,,\ \ i = 1,\ldots,m\,.
	\end{array} \right.
\]
We split the constraints of Problem~(P) into two sets: 
\begin{eqnarray} 
	&& {\cal A} := \big\{u\in {\cal L}^{2}([0,1];\dR^m)\ |\ \exists x\in W^{1,2}([0,1];\dR^n)\mbox{ which solves } \nonumber \\[1mm]
	&&\hspace*{45mm} \dot{x}(t) = A(t)\,x(t) + B(t)\,u(t)\,,\ \ \mbox{for a.e.\ } t\in[0,1]\,, \mbox{ and }\nonumber \\[1mm]
	&&\hspace*{45mm} \varphi(x(0),x(1)) = 0 \big\}\,, \label{A} \\[2mm]
	&& {\cal B} := \big\{u\in {\cal L}^{2}([0,1];\dR^m)\ |\ \underline{a}_i(t) \le u_i(t) \le \overline{a}_i(t)\,,\ \mbox{for a.e.\ } t\in[0,1]\,,\ i = 1,\ldots,m\big\}\,, \label{B}
\end{eqnarray}
where $W^{1,2}([0,1];\mathbb{R}^q)$ is the Sobolev space of absolutely continuous functions, namely,
\[W^{1,2}([0,1];\mathbb{R}^q):=\left\{z\in {\cal L}^2([0,1];\mathbb{R}^q)\,|\,\dot{z}:=dz/dt\in {\cal L}^2([0,1];\mathbb{R}^q)\right\},\]
endowed with the norm
\[\|z\|_{W^{1,2}}:=(\|z\|_{{\cal L}^2}^2+\|\dot{z}\|_{{\cal L}^2}^2)^{1/2}.\]
In the particular case where the objective function $f_0$ is quadratic in $x$ and $u$, Problem~(P) is called a {\em linear quadratic} (LQ) optimal control problem.
We assume that the control system $\dot{x}(t) = A(t)x(t) + B(t)u(t)$ is {\em controllable}---see the definition of {\em componentwise controllability} in Section~\ref{subsec:controllability}, which is a stronger version of controllability. Then there exists a (possibly not unique) $u(\cdot)$ such that the boundary-value problem given in ${\cal A}$ has a solution $x(\cdot)$.  In other words, ${\cal A} \neq \varnothing$.  Also, clearly, ${\cal B} \neq \varnothing$.  Recall that $\varphi$ is affine, so the constraint set~${\cal A}$ is an {\em affine subspace}. We note that by \cite[Corollary 1]{BurCalKay2024},
\begin{equation}  \label{eq:A closed}
{\cal A}\ \ \mbox{is closed}.  
\end{equation}
Given that ${\cal B}$ is a {\em box}, the constraint sets  ${\cal A}$ and ${\cal B}$ turn out to be two closed and convex sets in a Hilbert space. Using the fact that every sequence that converges in ${\cal L}^2$ has a subsequence that converges pointwise almost everywhere, it can be directly checked that ${\cal B}$ is closed in ${\cal L}^2(0,1;\dR^m)$.  It will be convenient to use the expression
\[
B(t)u(t) = \sum_{i=1}^m b_i(t)\,u_i(t)\,,
\]
where $b_i(t)$ is the $i$th column of the matrix $B(t)$, interpreted as the column vector associated with the $i$th control component $u_i$.
	
If ${\cal A} \cap {\cal B} \neq \varnothing$\,, then one has a {\em feasible} optimal control problem.  The {\em feasibility problem} is posed as one of finding an element in ${\cal A} \cap {\cal B}$, namely:
\begin{equation}  \label{prob:feas}
	\mbox{Find\ \ } u \in {\cal A} \cap {\cal B}\,.
\end{equation}
However, if ${\cal A} \cap {\cal B} = \varnothing$ then the problem is said to be {\em infeasible}.  Obviously, the feasibility problem in~\eqref{prob:feas} has no solution in this case, but in Section~\ref{subsec:best_approx} we will recall the problem of finding (in some sense) a {\em best approximation solution} and summarize the results obtained in~\cite{BurKayMou2024}.

\subsection{Problem of minimizing gap}

Given the LQ control problem and the definitions of the constraint sets ${\cal A}$ and ${\cal B}$ in~\eqref{A}--\eqref{B}, we consider the case where ${\cal A} \cap {\cal B} = \varnothing$.  
A {\em best approximation pair} $(u_{\cal A}^*, u_{\cal B}^*) \in {\cal A}\times{\cal B}$ is one that solves
\begin{equation}  \label{prob:min_dist}
	\min_{\substack{u_{\cal A}\in {\cal A} \\ u_{\cal B}\in {\cal B}}}\ \ \ds\frac{1}{2}\,\|u_{\cal A} - u_{\cal B}\|_{{\cal L}^2}^2\,,
\end{equation}
where $\|\cdot\|_{{\cal L}^2}$ is the ${{\cal L}^2}$ norm.  Define the {\em gap} ({\em function}) {\em vector} (see \cite{BauMou2017})
\begin{equation}  \label{gap_v}
	v := u_{\cal A} - u_{\cal B},\,
 \qquad (u_{\cal A}, u_{\cal B}) \in {\cal A}\times{\cal B}.
\end{equation}
Here, the ${\cal L}^2$-distance $\|u_{\cal A} - u_{\cal B}\|_{{\cal L}^2}$ is referred to as the {\em gap}.

Using $u_{\cal A} = v + u_{\cal B}$ and the definitions of ${\cal A}$ and ${\cal B}$, the problem of gap-minimization in~\eqref{prob:min_dist} can be rewritten in the format of a standard {\em optimal control problem} as follows.
\[
\mbox{(Pf) }\left\{\begin{array}{rl}
\ds\min_{v(\cdot),u_{\cal B}(\cdot)} & \ \ \ds\frac{1}{2}\int_0^1 \|v(t)\|_2^2\, dt \\[2mm] 
\mbox{subject to} 
& \ \ \ds\dot{x}(t) = A(t)x(t) + \sum_{i=1}^m b_i(t)(v_i(t) + u_{{\cal B},i}(t))\,,\ \ \mbox{for a.e.\ } t\in[0,1]\,, \\[2mm]
& \ \ \varphi(x(0),x(1)) = 0\,, \\[2mm]
& \ \ \underline{a}_i(t) \le u_{{\cal B},i}(t) \le \overline{a}_i(t)\,,\ \ \mbox{for a.e.\ } t\in[0,1]\,,\ \ i = 1,\ldots,m\,,
\end{array} \right.
\]
where $\|\cdot\|_2$ is the Euclidean norm.  Problem~(Pf) is an optimal control problem with two control variable vectors, namely $v$ and $u_{\cal B}$, where $v(t) := (v_1(t),\ldots,v_m(t))\in\dR^m$ and\linebreak $u_{\cal B}(t) := (u_{{\cal B},1}(t),\ldots,u_{{\cal B},m}(t))\in\dR^m$.

With the introduction of the new unconstrained (or free) control $v$, Problem~(Pf) is feasible.  We further note that Problem~(Pf), having a quadratic objective functional and linear equality and inequality constraints, is convex.

\subsection{Componentwise controllability}
\label{subsec:controllability}

We call the control system\ \  $\dot{x}(t) = A(t)x(t) + B(t)u(t)$\ \ {\em controllable w.r.t.\ $u_i$} on $[0,1]$ if given any initial state $x(0) = x_0$ there exists a continuous $i$th component $u_i(\cdot)$ of the control $u(\cdot)$ such that the corresponding solution of
\begin{equation}  \label{contr_sys_ui}
	\dot{x}(t) = A(t)x(t) + b_i(t)u_i(t)
\end{equation}
satisfies $x(1) = 0$.  Then clearly Theorems~1 and 2 cited in~\cite{BurKayMou2024} for a more general definition of controllability, with $B(t)u(t)$ replaced by $b_i(t)u_i(t)$, hold for the system in~\eqref{contr_sys_ui}.

If the control system is time-invariant, i.e., $A(t) = A$ and $b_i(t) = b_i$, for all $t\in[0,1]$, where $A$ and $b_i$ are constant matrices, then Theorem~2 cited in~\cite{BurKayMou2024} simplifies to checking the straightforward {\em Kalman controllability rank condition}\,:  the control system\ \  $\dot{x}(t) = A\,x(t) + B\,u(t)$\ \ is controllable w.r.t.\ $u_i$ on $[0,1]$ if $\rank\, [\ b_i\ |\  A\,b_i\ |\  \ldots\ |\ A^{n-1}\,b_i\ ] = n$.

\subsection{Best approximation control}
\label{subsec:best_approx}

Next, we provide the result~\cite[Theorem~3]{BurKayMou2024} for the best approximation solution in the set ${\cal B}$, in the case when the constraint sets ${\cal A}$ and ${\cal B}$ are disjoint. 
\begin{theorem}[Gap Vector and Best Approximation Control in \boldmath{${\cal B}$}~\cite{BurKayMou2024}]  \label{thm:gap&bbang}
With the notation of Problem (Pf), assume that ${\cal A} \cap {\cal B} = \varnothing$.  Then the optimal gap vector is given by
\begin{equation}  \label{eq:gap_soln}
v(t) = -B^T(t)\lambda(t)\,,
\end{equation}
for all $t\in[0,1]$, where the adjoint variable vector $\lambda(\cdot)$ solves\ \ $\dot{\lambda}(t) = -A^T(t)\,\lambda(t)$, with appropriate transversality conditions.  Moreover, suppose that $A(\cdot)$ and $B(\cdot)$ are sufficiently smooth and that the control system\ \ $\dot{x}(t) = A(t)x(t) + B(t)u(t)$\ \ is controllable w.r.t.\ $u_i$ on any $[t',t'']\subset[0,1]$, $t'<t''$, for some $i = 1,\ldots,m$.  Then, for a.e.\ $t\in[0,1]$,
\begin{equation}  \label{uB}
	u_{{\cal B},i}(t) = \left\{\begin{array}{rl}
	\overline{a}_i(t)\,, &\ \ \mbox{if\ \ } v_i(t) \ge 0\,, \\[1mm]
	\underline{a}_i(t)\,, &\ \ \mbox{if\ \ } v_i(t) < 0\,.
	\end{array} \right.
\end{equation} 
In other words, such $u_{{\cal B},i}$ is of bang--bang type.
\end{theorem}

The following corollary is restated from~\cite[Remark~1]{BurKayMou2024}.
\begin{corollary}[Best Approximation Control in \boldmath{${\cal A}$}~\cite{BurKayMou2024}] \rm
Suppose that the hypotheses of Theorem~\ref{thm:gap&bbang} hold.  Then the best approximation solution in the set ${\cal A}$ is given by
\begin{equation}  \label{uA}
	u_{{\cal A},i}(t) = \left\{\begin{array}{rl}
	\overline{a}_i(t) + v_i(t)\,, &\ \ \mbox{if\ \ } v_i(t) \ge 0\,, \\[1mm]
	\underline{a}_i(t) + v_i(t)\,, &\ \ \mbox{if\ \ } v_i(t) < 0\,,
	\end{array} \right.
\end{equation}
for a.e.\ $t\in[0,1]$.
\end{corollary}
We observe that while the component $v_i(\cdot)$, as can be deduced from~\eqref{eq:gap_soln}, is continuous, $u_{{\cal A},i}(\cdot)$ is piecewise continuous.

The following straightforward result is new.

\begin{corollary}[Projection of \boldmath{$u_{\cal B}$} onto \boldmath{${\cal A}$}] \rm
Suppose that $u_{\cal A}$ and $u_{\cal B}$ are the best approximation controls.  Then $u_{\cal A}$ is the projection of $u_{\cal B}$ onto ${\cal A}$; namely
\begin{equation}  \label{eq:proj_uB_A}
	u_{\cal A} = P_{\cal A}(u_{\cal B})\,.
\end{equation}
\end{corollary}
\begin{proof}
From the definition of $v$ in~\eqref{gap_v}, we have $u_{\cal A} = u_{\cal B} + v$. Substituting next the optimal gap solution in~\eqref{eq:gap_soln}, one gets $u_{\cal A} = u_{\cal B} - B^T(t)\lambda$, where $\lambda(\cdot)$ solves\ \ $\dot{\lambda}(t) = -A^T(t)\,\lambda(t)$, with appropriate transversality conditions.  Subsequently, Theorem~3.4 and its proof in~\cite{BurCalKay2024} furnish~\eqref{eq:proj_uB_A}.
\end{proof}

One further observation is in order for the special but common case of the time-invariant linear control system $\dot{x}(t) = A\,x(t) + B\,u(t)$: if\ \ $\rank\, [\ b_i\ |\  A\,b_i\ |\  \ldots\ |\ A^{n-1}\,b_i\ ] = n$, then the best approximation control component $u_{{\cal B},i}$ for the infeasible optimal control problem is of bang--bang type as given in \eqref{uB}.

\section{A Double Integrator Problem}
\label{sec:DI}

We consider the following optimal control problem involving the double integrator as a special case of Problem~(P).
	\[
	\mbox{(PDI) }\left\{\begin{array}{rl}
		\ds\min & \ \ \ds\frac{1}{2}\int_0^1 u^2(t)\,dt =  \frac{1}{2}\,\|u\|_{{\cal L}^2}^2 \\[5mm] 
		\mbox{subject to} & \ \ \dot{x}_1(t) = x_2(t)\,,\ \ x_1(0) = s_0\,,\ \ x_1(1) = s_f\,, \\[2mm]
		& \ \ \dot{x}_2(t) = u(t)\,,\ \ \ \,x_2(0) = v_0\,,\ \ x_2(1) = v_f\,, \\[2mm]
		& \ \ |u(t)|\le a\,,\ \ \mbox{for all}\ \ t\in [0,1]\,,
	\end{array} \right.
	\]
where the bound $a>0$ and the boundary values $s_0$, $s_f$, $v_0$ and $v_f$ are specified. Problem~(PDI) was previously used in~\cite{BauBurKay2019, BurCalKayMou2024} to study applications of splitting and projection methods, but only in a feasible setting.  The above problem was also studied in~\cite{BurKayMou2024} for the infeasible case, i.e. the case where the control bound $a$ is small enough to make the constraint set infeasible.  However, no splitting and projection methods were implemented in~\cite{BurKayMou2024}.  In~\cite{BurKayMou2024}, the critically feasible case was also studied, namely the case of the smallest $a$ for which Problem~(PDI) is (still) feasible. 
In (PDI), we have $A = \begin{bmatrix}
0 & 1 \\
0 & 0
\end{bmatrix}$ and $b = \begin{bmatrix}
0 \\
1
\end{bmatrix}$.  So, $\rank\,[\,b\ |\ Ab\,] = \rank\,\begin{bmatrix}
0 & 1 \\
1 & 0
\end{bmatrix} = 2 = n$, and therefore the control system is controllable.  
Recall that the adjoint variables satisfy $\dot{\lambda}(t) = -A^T\lambda(t)$, i.e. $\dot{\lambda}_1(t) = 0$ and $\dot{\lambda}_2(t) = -\lambda_1$, with no boundary (or transversality) conditions imposed in this particular case.  It follows that $\lambda_1(t) = c_1$ and $\lambda_2(t) = -c_1\,t - c_2$, for all $t\in[0,1]$, where $c_1$ and $c_2$ are real constants.  So, by Equation~\eqref{eq:gap_soln} in Theorem~\ref{thm:gap&bbang}, the gap function is simply given by
\[
v(t) = c_1\,t + c_2\,,
\]
for all $t\in[0,1]$.  Consequently,
\[
u_{\cal A}(t) = u_{\cal B}(t) + c_1\,t + c_2\,.
\]
Combining the above fact with \eqref{uB} in Theorem \ref{thm:gap&bbang} we obtain
\begin{equation}  \label{uB_DI}
	u_{\cal B}(t) = \left\{\begin{array}{rl}
	a\,, &\ \ \mbox{if\ \ } c_1\,t + c_2 \ge 0\,, \\[1mm]
	-a\,, &\ \ \mbox{if\ \ } c_1\,t + c_2 < 0\,,
	\end{array} \right.
\end{equation} 
for a.e.\ $t\in[0,1]$.

\begin{remark}[Types of Control According to the Gap Function]  \label{rem:uB_DI}  \rm
The gap function\linebreak $v(t) = c_1\,t + c_2$ is affine and thus can obviously change sign at most once.  So, $u_{\cal B}$ is constant or piecewise-constant, depending on the values of $c_1$ and $c_2$, as elaborated in the following.
\begin{itemize}
\item[(A)] $c_1 = 0$ or $c_2 = 0$\,: For all $t\in[0,1]$, $u_{\cal B}(t) = \sgn(c_1)\,a$ if $c_2 = 0$ and $u_{\cal B}(t) = \sgn(c_2)\,a$ if $c_1 = 0$.  In other words, $u_{\cal B}(t) = a$ or $-a$, for all $t\in[0,1]$.  
\item[(B)] $c_1 \neq 0$ and $c_2 \neq 0$\,: We have two subcases to examine.  
\begin{itemize}
\item[(I)] $c_1$ and $c_2$ have the same sign:  Then $u_{\cal B}(t) = \sgn(c_2)\,a$, i.e., $u_{\cal B}(t) = a$ or $-a$, for all $t\in[0,1]$.  

\item[(II)] $c_1$ and $c_2$ have opposite signs:
\begin{itemize}
\item[(i)] $v(0) = c_2$ and $v(1) = c_1 + c_2$ have the same sign: Then $u_{\cal B}(t) = \sgn(c_2)\,a$, i.e., $u_{\cal B}(t) = a$ or $-a$, for all $t\in[0,1]$.
\item[(ii)] $v(0) = c_2$ and $v(1) = c_1 + c_2$ have opposite signs:  Then, by \eqref{uB_DI},
\begin{equation}  \label{uB_DI_a}
	u_{\cal B}(t) = \left\{\begin{array}{rl}
	r\,, &\ \ \mbox{if\ \ } 0\le t < t_s\,, \\[1mm]
	-r\,, &\ \ \mbox{if\ \ } t_s\le t \le 1\,,
	\end{array} \right.
\end{equation}
where
\[
r = \sgn(v(0))\,a = \sgn(c_2)\,a\,.
\]
In~\eqref{uB_DI_a}, the value of $t_s\in(0,1]$ such that $v(t_s) = 0$ is called the {\em switching time}.  When $v(1) = 0$ we set $t_s = 1$ in~\eqref{uB_DI_a}.  Although the value of $u_{\cal B}$ does not change when $t_s = 1$, we still refer to $t_s = 1$ here (with $v(1) = 0$) as a switching time.  On the other hand, we cannot set $t_s = 0$, as then $v(0) = 0$ implies that $c_2 = 0$, which contradicts the condition that $c_1$ and $c_2$ are nonzero.
\proofbox
\end{itemize}
\end{itemize}
\end{itemize}

\end{remark}
\begin{remark}[\boldmath{Setting $t_s = 1$} in Certain Cases]  \label{rem:uB_DI_ts=1}  \rm
Although $\sgn(c_1) = \sgn(c_2)$ in Remark~\ref{rem:uB_DI} (B)(I), $u_{\cal B}$ will still be the same constant {over the whole interval $[0,1]$} if we set $c_1 := -c_2$, as a result of which, $v(1) = 0$, or $t_s = 1$.  We also note, in the case of (B)(II)(i), that $u_{\cal B}$ will remain {to be} the same {constant}, if we again set $c_1 := -c_2$, and so $v(1) = 0$, or $t_s = 1$.  Therefore, in both Remark~\ref{rem:uB_DI}(B)(I) and (B)(II)(i), $u_{\cal B}$ can be represented by~\eqref{uB_DI_a} with $t_s = 1$.  This will be used in the proof of Theorem~\ref{DI_infeas_control} for the case where $c_1\neq 0\neq c_2$.
\proofbox
\end{remark}

Let $a_c$ be the critical value of $a$ such that Problem~(PDI) is feasible when $a = a_c$ but infeasible when $a < a_c$.  Therefore, when $a = a_c$ Problem~(PDI) is said to be {\em critically feasible}.  The critical value $a_c$ is expressed in~\cite[Theorem~5]{BurKayMou2024}, by cases, where $t_c$ stands for the (critical) switching time, as follows.
\begin{fact}[Theorem~5 in~\cite{BurKayMou2024}] \label{fact_ac}
The critical value $a_c$ for Problem~(PDI) is given in one of the following two cases.
\begin{itemize}
\item[(a)]  $s_f - s_0 \neq (v_0 + v_f)/2$\,:
\begin{itemize}
\item[(i)] $v_0 \neq v_f$\,: $a_c = \ds\left|\frac{v_f - v_0}{2 t_c - 1}\right|$\,, \\ where $t_c$ solves\ \ $\ds (v_f - v_0)\,t_c^2 + 2\,(s_f - s_0 - v_f)\, t_c + \frac{1}{2}\,(v_0 + v_f) - (s_f - s_0) = 0$\,.

\item[(ii)] $v_0 = v_f$\,:  $a_c = 4\,|v_f + s_0 - s_f|$.
\end{itemize}
\item[(b)]  $s_f - s_0 = (v_0 + v_f)/2$\,: $a_c = |v_f - v_0|$, with $v_0 \neq v_f$.
\end{itemize}
\end{fact}

\begin{theorem}[Best Approximation Solution to Infeasible Problem~(PDI)] \label{DI_infeas_control} \ 
Suppose that $a < a_c$, with the critical feasibility value $a_c$ as given in Fact~\ref{fact_ac}.
\begin{enumerate}
\item[(a)]  If $s_f - s_0 \neq (v_0 + v_f)/2$, then the best approximation control is given by
\begin{equation}  \label{uB_DI1}
	u_{\cal B}(t) = \left\{\begin{array}{rl}
	r\,, &\ \ \mbox{if\ \ } 0\le t < t_s\,, \\[1mm]
	-r\,, &\ \ \mbox{if\ \ } t_s\le t \le 1\,,
	\end{array} \right.
\end{equation}
with $r$ and the switching time $t_s$ given in the following two cases.
\begin{enumerate}
\item[(i)] $v_0 \neq v_f$\,:
\begin{equation}  \label{r1}
	r = -\sgn(c_1)\,a\,
\end{equation}
and $c_1$ and $t_s$ solve the following system of equations.
\begin{eqnarray}  \label{c1_ts}
&& (2r - c_1)\,t_s + \frac{1}{2}\,c_1 - v_f + v_0 - r = 0\,, \label{eqn1}\\
&& (r - c_1)\,t_s^2 + (v_0 - v_f + c_1 - r)\,t_s - \frac{1}{3}\,c_1 + \frac{1}{2}\,r +v_f + s_0 - s_f = 0\,.  \label{eqn2}
\end{eqnarray}
\item[(ii)] $v_0 = v_f$\,:
\begin{equation}  \label{r2}
	r = \left\{\begin{array}{rl}
	a\,, & \ \mbox{if\ \ } 4(v_f + s_0 - s_f) < -a\,, \\[1mm]
	-a\,, & \ \mbox{if\ \ } 4(v_f + s_0 - s_f) > a\,,
	\end{array} \right.
	\quad\mbox{and}\quad t_s = \frac{1}{2}\,.
\end{equation}
\end{enumerate}
			
\item[(b)]  If $s_f - s_0 = (v_0 + v_f)/2$, then $v_0 \neq v_f$ and the best approximation control is constant and given by
\begin{equation}  \label{uB_DI2}
	u_{\cal B}(t) = \left\{\begin{array}{rl}
	a\,, & \ \mbox{if\ \ } v_f - v_0 > a\,, \\[1mm]
	-a\,, & \ \mbox{if\ \ } v_f - v_0 < -a\,,
\end{array} \right.
\end{equation}
for all $t\in [0,1]$.
\end{enumerate}
\end{theorem}
\begin{proof}
Since $a < a_c$, i.e. ${\cal A} \cap {\cal B} = \varnothing$, the gap function $v \neq 0$, i.e. $v(t) = c_1\,t + c_2 \neq 0$ for some $t\in[0,1]$.  Therefore, we exclude the case where $c_1 = c_2 = 0$ and examine the following three cases. \\[-3mm]
		
\centerline{(I) $c_1 \neq 0$ and $c_2 = 0$;\ \ \ (II) $c_1 = 0$ and $c_2 \neq 0$;\ \ \ and\ \ \ (III) $c_1 \neq 0$ and $c_2 \neq 0$.}
		
\noindent
Case~(I): Suppose that $c_1 \neq 0$ and $c_2 = 0$.  Let $r = -\sgn(c_1)\,a$.  Then, from~\eqref{uB_DI},
\begin{equation} \label{uB_case_I}
u_{\cal B}(t) = -r\,, \mbox{ for all } t\in[0,1]\,,
\end{equation}
and so $u_{\cal A}(t) = -r + c_1\,t$, for all $t\in[0,1]$.
The state equations of Problem~(PDI) with this $u_{\cal A}(t)$ substituted, namely $\dot{x}_2(t) = -r + c_1\,t$ and $\dot{x}_1(t) = x_2(t)$, produce the solutions $x_2(t) = v_0 - r\,t + c_1\,t^2/2$ and $x_1(t) = s_0 + v_0\,t - r\,t^2/2 + c_1\,t^3/6$.  Subsequently, by imposing the boundary conditions $x_2(1) = v_f$ and $x_1(1) = s_f$, and re-arranging, we obtain the equations
\begin{eqnarray}
r &=&  \frac{1}{2}\,c_1 - v_f + v_0 \,,  \label{eqn1_case_I} \\
s_0 - s_f &=& \frac{1}{2}\,r - v_0 - \frac{1}{6}\,c_1\,.  \label{eqn2_case_I}
\end{eqnarray}
The substitution of $r$ in \eqref{eqn1_case_I} into \eqref{eqn2_case_I} and manipulations result in
\begin{equation} \label{eq:c1}
c_1 = 6\,(v_0 + v_f - 2\,(s_f - s_0))\,.
\end{equation}
Since $c_1 \neq 0$, we must have from \eqref{eq:c1} that $s_f - s_0 \neq (v_0 + v_f)/2$, which is the hypothesis posed in part~(a) of the theorem.  Since $r = -\sgn(c_1)\,a$, $c_1$ can also be solved directly from \eqref{eqn1_case_I} as
\begin{equation}  \label{c1_a}
	c_1 = \left\{\begin{array}{rl}
	2\,(v_f - v_0 - a)\,, & \ \mbox{if\ \ } c_1 > 0\,, \\[1mm]
	2\,(v_f - v_0 + a)\,, & \ \mbox{if\ \ } c_1 < 0\,,
	\end{array} \right.
\end{equation}
i.e.,
\begin{equation}  \label{c1_a2}
	c_1 = \left\{\begin{array}{rl}
	2\,(v_f - v_0 - a)\,, & \ \mbox{if\ \ } v_f - v_0 > a\,, \\[1mm]
	2\,(v_f - v_0 + a)\,, & \ \mbox{if\ \ } v_f - v_0 < -a\,,
	\end{array} \right.
\end{equation}
which implies that $a < |v_f - v_0|$ and that $v_0 \neq v_f$, which is the case in part~(a)(i).  Since $v(0) = 0$, we set $t_s = 0$ and observe that $u_{\cal B}(t)$ in~\eqref{uB_case_I} verifies \eqref{uB_DI1} for $t_s=0$.  Furthermore, with $t_s = 0$, Equations~\eqref{eqn1}--\eqref{eqn2} can be shown (after straightforward algebraic manipulations) to reduce to Equations~\eqref{eqn1_case_I}--\eqref{eqn2_case_I}.  So, we have shown that Case~(I) corresponds to and verifies part~(a)(i) of the theorem.
	
\noindent
Case~(II): Suppose that $c_1 = 0$ and $c_2 \neq 0$.  Then, from~\eqref{uB_DI}, $u_{\cal B}(t) = r$, where $r = \sgn(c_2)\,a$, and thus $u_{\cal A}(t) = r + c_2$, for all $t\in[0,1]$.  With substitution of $u_{\cal B}(t)$, the state equations become $\dot{x}_2(t) = r + c_2$ and $\dot{x}_1(t) = x_2(t)$, the solution of which results in $x_2(t) = v_0 + (r + c_2)\,t$ and $x_1(t) = s_0 + v_0\,t + (r + c_2)\,t^2/2$.  Applying the boundary conditions $x_2(1) = v_f$ and $x_1(1) = s_f$ one gets the equations
\begin{eqnarray}
r &=&  -c_2 + v_f - v_0 \,,  \label{eqn1_case_II} \\
s_f - s_0 &=& \frac{1}{2}\,(r + c_2) + v_0\,.  \label{eqn2_case_II}
\end{eqnarray}
Substituting $r$ in \eqref{eqn1_case_II} into \eqref{eqn2_case_II} and rearranging, we get the condition $s_f - s_0 = (v_f + v_0)/2$ given in part (b) of the theorem.  Since $r = \sgn(c_2)\,a$, $c_2$ can be solved directly from \eqref{eqn1_case_II} for $v_0 \neq v_f$ as
\begin{equation}  \label{c1_a3}
	c_2 = \left\{\begin{array}{rl}
	v_f - v_0 - a\,, & \ \mbox{if\ \ } v_f - v_0 > a\,, \\[1mm]
	v_f - v_0 + a\,, & \ \mbox{if\ \ } v_f - v_0 < -a\,,
	\end{array} \right.
\end{equation}
which implies that $a < a_c = |v_f - v_0|$.  Since $u_{\cal B}(t) = \sgn(c_2)\,a$, the expression in~\eqref{uB_DI2} follows.  In this case, one cannot have $v_0 = v_f$ as then \eqref{eqn1_case_II} reduces to $\sgn(c_2)\,a = -c_2$, or $|c_2| = -a < 0$, which is not possible.  So, we have shown that Case~(II) corresponds to and verifies part~(b) of the theorem.
		
\noindent
Case~(III): Finally, suppose that $c_1 \neq 0$ and $c_2 \neq 0$.  By Remark~\ref{rem:uB_DI}, case (B),  we have only two possible structures for the solution {(also see Remark~\ref{rem:uB_DI_ts=1})}: the first is the one given by~\eqref{uB_DI_a} in Remark~\ref{rem:uB_DI}(B)(I) or Remark~\ref{rem:uB_DI}(B)(II)(i), and the second is $u_{\cal B}(t) = \sgn(c_2)\,a$, for all $t\in[0,1]$. 

Let us consider the first solution, i.e., $u_{\cal B}(t)$ in~\eqref{uB_DI_a}.  Substitute $u(t) = u_{\cal A}(t) = u_{\cal B}(t) + c_1\,t + c_2$, into the differential equations in Problem~(PDI) to get
\begin{equation}  \label{x2_ODE}
	\dot{x}_2(t) = \left\{\begin{array}{rl}
	r + c_1\,t + c_2\,, & \ \mbox{if\ \ } 0\le t < t_s\,, \\[1mm]
	-r + c_1\,t + c_2\,, & \ \mbox{if\ \ } t_s \le t \le 1\,,
	\end{array} \right.
\end{equation}
for $t_s\in(0,1]$, where $r = \sgn(c_2)\,a$, and $\dot{x}_1(t) = x_2(t)$.  
\noindent

For $0\le t < t_s$, straightforward integration of the state equations with the initial conditions $x_2(0) = v_0$ and $x_1(0) = s_0$ gives the solution
\begin{eqnarray*}
&& x_2(t) = \frac{1}{2}\,c_1\,t^2 + (r+c_2)\,t + v_0\,, \\ 
&& x_1(t) = \frac{1}{6}\,c_1\,t^3 + \frac{1}{2}\,(r+c_2)\,t^2 + v_0\,t + s_0\,.
\end{eqnarray*}
For $t_s\le t \le 1$ and the terminal values $x_2(1) = v_f$ and $x_1(1) = s_f$, the solution of the state equations, after some lengthy algebraic manipulations, results in
\begin{eqnarray*}
&& x_2(t) = \frac{1}{2}\,c_1\,(t^2 - 1) + (r-c_2)\,(1 - t) + v_f\,, \\ 
&& x_1(t) = \frac{1}{6}\,c_1\,(t + 2)\,(1 - t)^2 - \frac{1}{2}\,(r-c_2)\,(1 - t)^2 - v_f\,(1 - t) + s_f\,.
\end{eqnarray*}
Since $x_i$ are continuous, $\lim_{t\to t_s^-} x_i(t) = \lim_{t\to t_s^+} x_i(t)$, $i = 1,2$.  In other words,
\[
	\frac{1}{2}\,c_1\,t_s^2 + (r+c_2)\,t_s + v_0 = \frac{1}{2}\,c_1\,(t_s^2 - 1) + (r-c_2)\,(1 - t_s) + v_f\,,
\]
\[
	\frac{1}{6}\,c_1\,t_s^3 + \frac{1}{2}\,(r+c_2)\,t_s^2 + v_0\,t_s + s_0 = \frac{1}{6}\,c_1\,(t _s+ 2)\,(1 - t_s)^2 - \frac{1}{2}\,(r-c_2)\,(1 - t_s)^2 - v_f\,(1 - t_s) + s_f\,.
\]
Cancellations and re-arrangements simplify these equations to
\begin{equation}  \label{eqn:x2a}
	2\,r\,t_s = -\frac{1}{2}\,c_1 + r - c_2 + v_f - v_0\,,
\end{equation}
\begin{equation}  \label{eqn:x1a}
	r\,t_s^2 = \frac{1}{6}\,c_1\,(2 - 3\,t_s) - \frac{1}{2}\,r\,(1 - 2\,t_s) + \frac{1}{2}\,c_2\,(1 - 2\,t_s) + (v_f - v_0)\,t_s + s_f - s_0 - v_f\,.
\end{equation}

By definition, at a switching time $t_s$, one has $v(t_s) = c_1\,t_s + c_2 = 0$, i.e., $c_2 = -c_1\,t_s$, which can be used to eliminate $c_2$ from \eqref{eqn:x2a}--\eqref{eqn:x1a}.  Then, further algebraic manipulations give rise to~\eqref{eqn1}--\eqref{eqn2}.  We also note that $r = \sgn(c_2)\,a = -\sgn(c_1)\,a$, as in part~(a)(i) of the theorem.  

Consider the special case where $v_0 = v_f$, for which \eqref{eqn:x2a} becomes
\begin{equation*}  \label{eq:v0=vf}
(2\,r - c_1) \left(t_s - \frac{1}{2}\right) = 0\,.
\end{equation*}
Here, $(2\,r - c_1) = (-2\,\sgn(c_1)\,a - c_1) \neq 0$ since $a > 0$.  Therefore $t_s = 1/2$.  Substituting this into~\eqref{eqn:x1a}, direct manipulations yield $c_1 = -3\,\sgn(c_1)\,a + 12\,(v_f + s_0 - s_f)$, i.e.,
\begin{equation}  \label{c2_a}
	c_1 = \left\{\begin{array}{rl}
	-3\,a + 12\,(v_f + s_0 - s_f)\,, & \ \mbox{if\ \ } c_1 > 0\,, \\[1mm]
	3\,a + 12\,(v_f + s_0 - s_f)\,, & \ \mbox{if\ \ } c_1 < 0\,.
	\end{array} \right.
\end{equation}
In other words, since $r = -\sgn(c_1)\,a$, we ultimately get, for the case where $v_0 = v_f$,
\begin{equation*}
	r = \left\{\begin{array}{rl}
	a\,, & \ \mbox{if\ \ } 4(v_f + s_0 - s_f) < -a\,, \\[1mm]
	-a\,, & \ \mbox{if\ \ } 4(v_f + s_0 - s_f) > a\,,
	\end{array} \right.
	\quad\mbox{and}\quad t_s = \frac{1}{2}\,.
\end{equation*}
which is the result required in part~(a)(ii) of the theorem.

Now consider the second solution, given by Remark~\ref{rem:uB_DI}(B)(I) and Remark~\ref{rem:uB_DI}(B)(II)(i), i.e. $u_{\cal B}(t) = r$, where $r := \sgn(c_2)\,a$,  for all $t\in[0,1]$.  Then $u_{\cal A}(t) = r + c_1\,t + c_2$, and so $\dot{x}_2(t) = r + c_1\,t + c_2$ and $\dot{x}_1(t) = x_2(t)$, resulting in $x_2(t) = v_0 + (r + c_2)\,t + c_1\,t^2/2$ and $x_1(t) = s_0 + v_0\,t + (r + c_2)\,t^2/2 + c_1\,t^3/6$.  Using $x_2(1) = v_f$ and $x_1(1) = s_f$ one gets
\begin{eqnarray}
&& r + c_2 = -\frac{1}{2}\,c_1 + v_f - v_0\,, \label{eqn:x2c} \\ 
&& s_f - s_0 = \frac{1}{6}\,c_1 + \frac{1}{2}\,(r+c_2) + v_0\,.  \label{eqn:x1c}
\end{eqnarray}
Substituting $r + c_2$ from \eqref{eqn:x2c} into \eqref{eqn:x1c} and rearranging, we get $c_1 = 6\,(v_0 + v_f - 2\,(s_f - s_0))$.  Since $c_1 \neq 0$,  $s_f - s_0 \neq (v_0 + v_f )/2$, which is the condition given for part~(a) of the theorem.  By Remark~\ref{rem:uB_DI_ts=1}, one can set $c_2 := -c_1$ and observe that one still gets $c_1 = 6\,(v_0 + v_f - 2\,(s_f - s_0))$.  The assignment $c_2 := -c_1$ also implies $\sgn(c_2) = -\sgn(c_1)$ and $t_s = 1$.  Then \eqref{eqn:x2c}--\eqref{eqn:x1c} become
\begin{eqnarray}
&& r = \frac{1}{2}\,c_1 + v_f - v_0\,, \label{eqn:x2d} \\ 
&& s_f - s_0 = -\frac{1}{3}\,c_1 + \frac{1}{2}\,r + v_0\,.  \label{eqn:x1d}
\end{eqnarray}
Note that the solution for $c_1$ from \eqref{eqn:x2d}--\eqref{eqn:x1d} is the same, i.e., $c_1 = 6\,(v_0 + v_f - 2\,(s_f - s_0))$, as that obtained from \eqref{eqn:x2c}--\eqref{eqn:x1c}.  Moreover, with $t_s = 1$ substituted, Equations~\eqref{eqn1}--\eqref{eqn2} reduce to Equations~\eqref{eqn:x2d}--\eqref{eqn:x1d}.  So, we have shown that Case~(III) corresponds to and verifies part~(a)(i) of the theorem.
\end{proof}

It is interesting to know if the value of the switching time $t_s$ converges to some limit as $a$ gets really small.  The following result informs us of that.

\begin{corollary}[Switching Time as \boldmath{$a \to 0$}]  \label{cor:a=0} \ \\
Suppose that $s_f - s_0 \neq (v_0 + v_f )/2$.  Then, as $a \to 0$,
\begin{equation}  \label{eq:t_s}
t_s \to \frac{2\,v_0 + v_f - 3\,(s_f - s_0)}{3\,(v_0 + v_f) - 6\,(s_f - s_0)}\,.
\end{equation}
\end{corollary}
\begin{proof}
Taking the limit as $a \to 0$, we have $r \to 0$ and \eqref{eqn1} becomes $-c_1\,t_s = v_f - v_0 - c_1/2$, which gives $c_1 = 2\,(v_f - v_0) / (1 - 2\,t_s)$.  Substitution of this expression for $c_1$ into \eqref{eqn2} and lengthy but straightforward algebraic manipulations result in~\eqref{eq:t_s}.
\end{proof}

\subsection{A case study}
Suppose that $s_0 = s_f = v_f = 0$ and $v_0 = 1$, as in the numerical example studied in~\cite{BurKayMou2024}.  For the critically feasible case, it was proved in~\cite[Remark~6]{BurKayMou2024} that $a_c = 1+\sqrt{2}$ with the switching time $t_c = 1/\sqrt{2}$.  The gap function for this case is $v(t) = 0$ for all $t\in[0,1]$, of course.  For the infeasible case, $a = 2$, $1.5$ and $1$ were considered and the switching time $t_s$ for each value of $a$ was computed, correct to three decimal places, by solving a direct discretization of Problem~(Pf).  These results from~\cite{BurKayMou2024} are listed in the last column of Table~\ref{table:DI_ts}.  In~\cite{BurKayMou2024}, the gap function for the infeasible case was constructed only graphically.

Theorem~\ref{DI_infeas_control} furnishes the means to compute $t_s$ and $v(t) = c_1\,t + c_2$ accurately and far more quickly than direct discretization: we just have to solve the system of equations \eqref{eqn1}--\eqref{eqn2} for $t_s$ and $c_1$, and use $c_2 = -c_1\,t_s$.  The solution for each value of $a$ is provided in Table~\ref{table:DI_ts}, computed correct to 15 decimal places by using the Newton method, after setting $r = -a$ in~\eqref{eqn1}--\eqref{eqn2}.  Since for each infeasible case $\sgn(c_2) < 0$, one has the control sequence $\{-a, a\}$, i.e., $u(t) = -a$ for $t\in[0,t_s)$ and $u(t) = a$ for $t\in[t_s,1]$.

\begin{table}[t!]
\centering
{\small
\begin{tabular}{ccccc} \hline \\[-2mm]
$a$ & $c_1$ & $c_2$ & $t_s$ & $t_s$~\cite{BurKayMou2024} \\[2mm] \hline \\[-3mm]
$1+\sqrt{2}$ & 0 & 0 &  &  $1/\sqrt{2}$ \\[1mm] 
2 & 0.971167995377200 & -0.680944121562971 & 0.701159969031407 & 0.701  \\[1mm] 
1.5 & 2.172720275813312 & -1.506394502798288 & 0.693321878369456 & 0.693  \\[1mm] 
1 & 3.410002343912741 & -2.335315565026760 & 0.684842803464806 & 0.685  \\[1mm] 
0.5 & 4.685608293782471 & -3.166921471137164 & 0.675882675754071 &   \\[1mm] 
0.1 & 5.734078050140310 & -3.833335390181957 & 0.668518174440991 &   \\[1mm] 
$\to 0$ & 6 & -4 & 2/3 & 0.667  \\[1mm] 
\hline\hline \end{tabular}
}
\caption{\sf Solutions for $c_1$, $c_2$ and $t_s$ by solving~\eqref{eqn1}--\eqref{eqn2}.  The last column lists the approximations of $t_s$ computed in~\cite{BurKayMou2024}.}
\label{table:DI_ts}
\end{table}

For this case study, when $a \to 0$ we get $t_s \to 2/3$ by Corollary~\ref{cor:a=0}, which was predicted numerically in~\cite{BurKayMou2024}.  Moreover, using the expression for $c_1$ in the proof of Corollary~\ref{cor:a=0}, in the limit $c_1 = 6$ and $c_2 = -c_1\,t_s = -4$, i.e., $v(t) = 6\,t - 4$.  These are also listed in Table~\ref{table:DI_ts}, both for illustration and for completeness.

Figure~\ref{fig:DI} shows for various values of $a$ the best approximation control and the gap function, all from~\cite{BurKayMou2024}.  Note again that $a = 1 + \sqrt{2}$ is the critically feasible case, while all other cases are infeasible.

\begin{figure}[t!]
\centering
\begin{subfigure}{.5\textwidth}
\centering
\includegraphics[width=\textwidth]{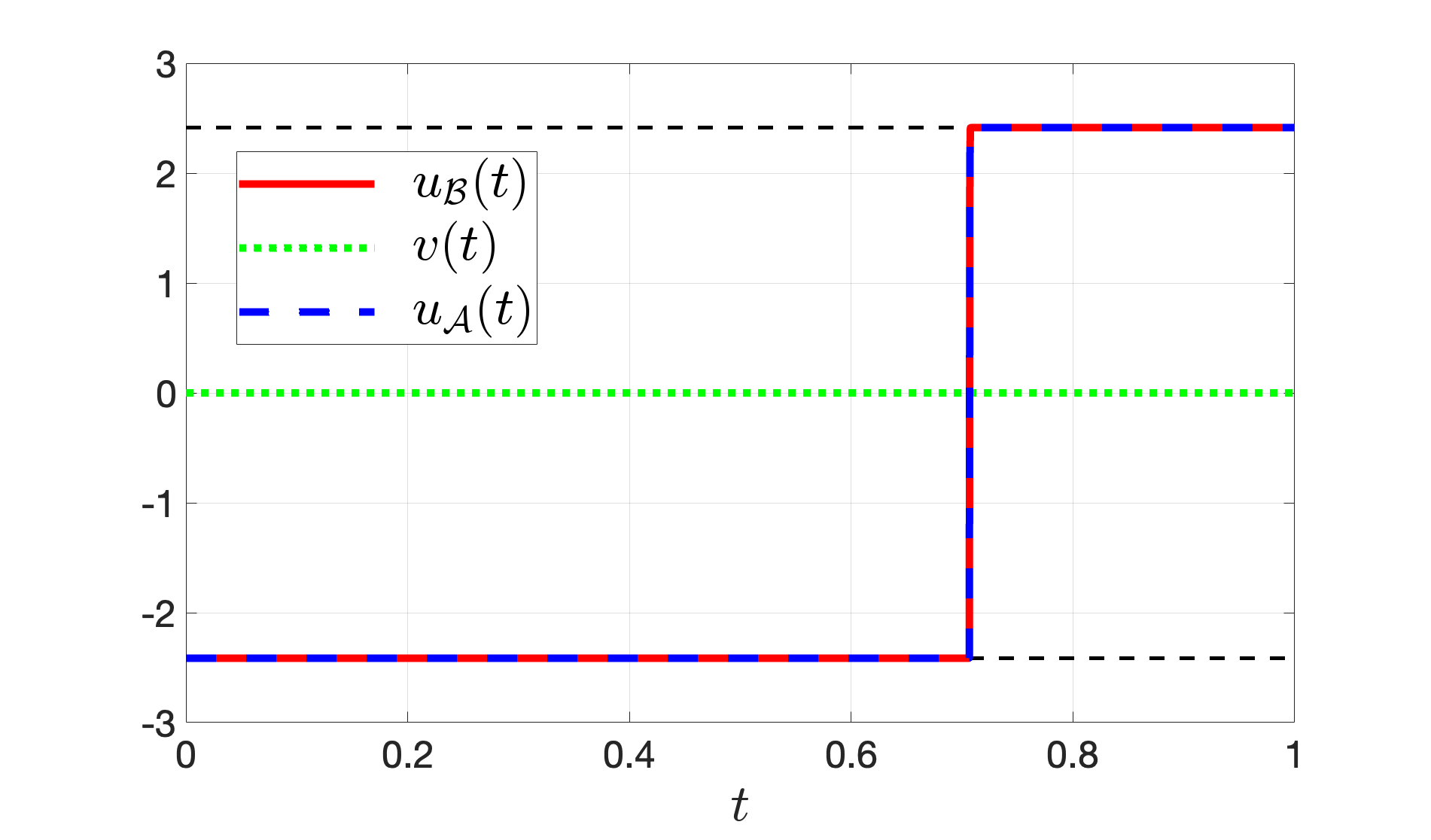}
\caption{$a = a_c = 1 + \sqrt{2}$}
\label{fig:a=2.414}
\end{subfigure}
\hfill
\begin{subfigure}{.49\textwidth}
\centering
\includegraphics[width=\textwidth]{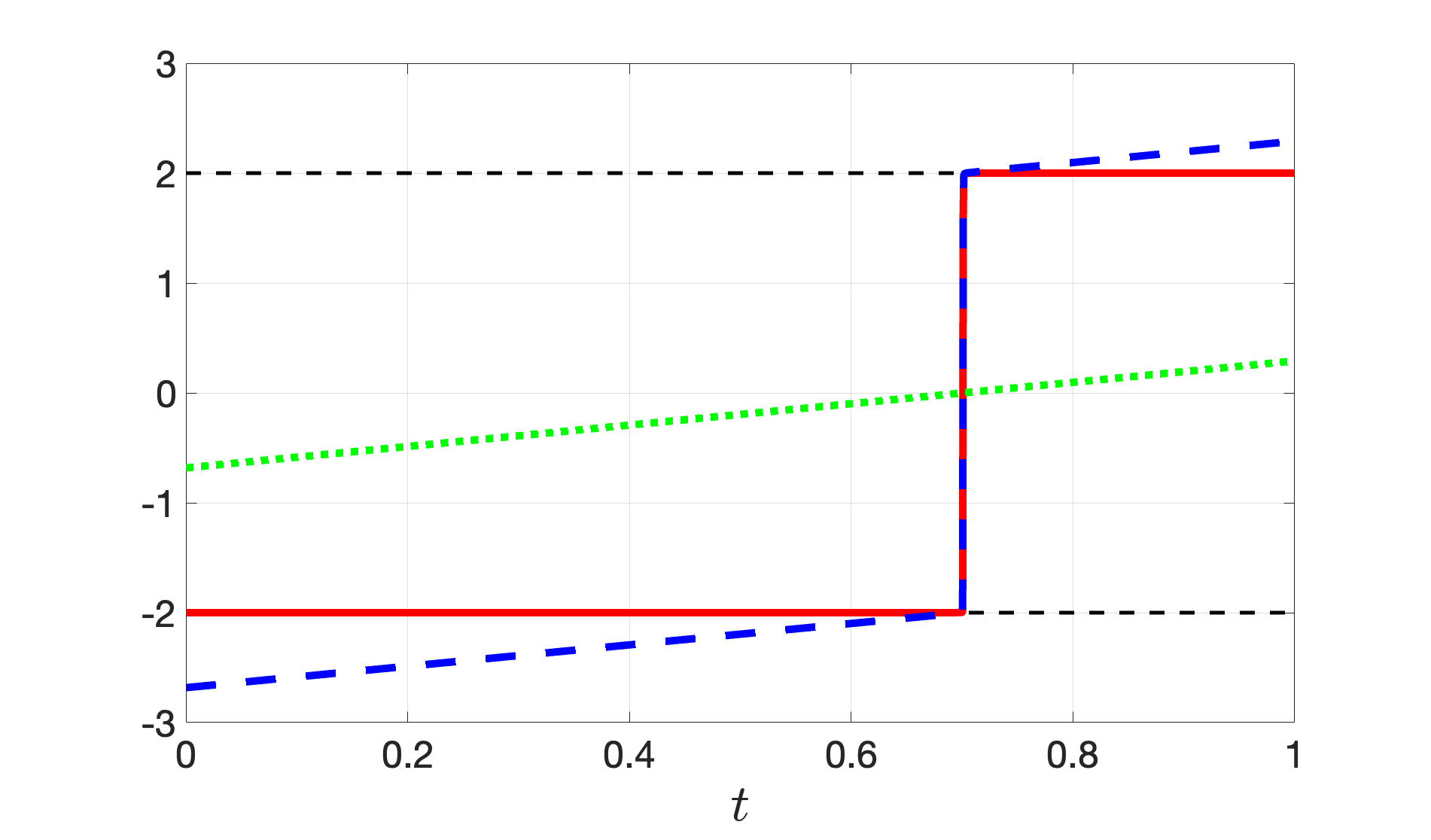}
\caption{$a = 2$}
\label{fig:a=2}
\end{subfigure}
\begin{subfigure}{.5\textwidth}
\centering
\includegraphics[width=\textwidth]{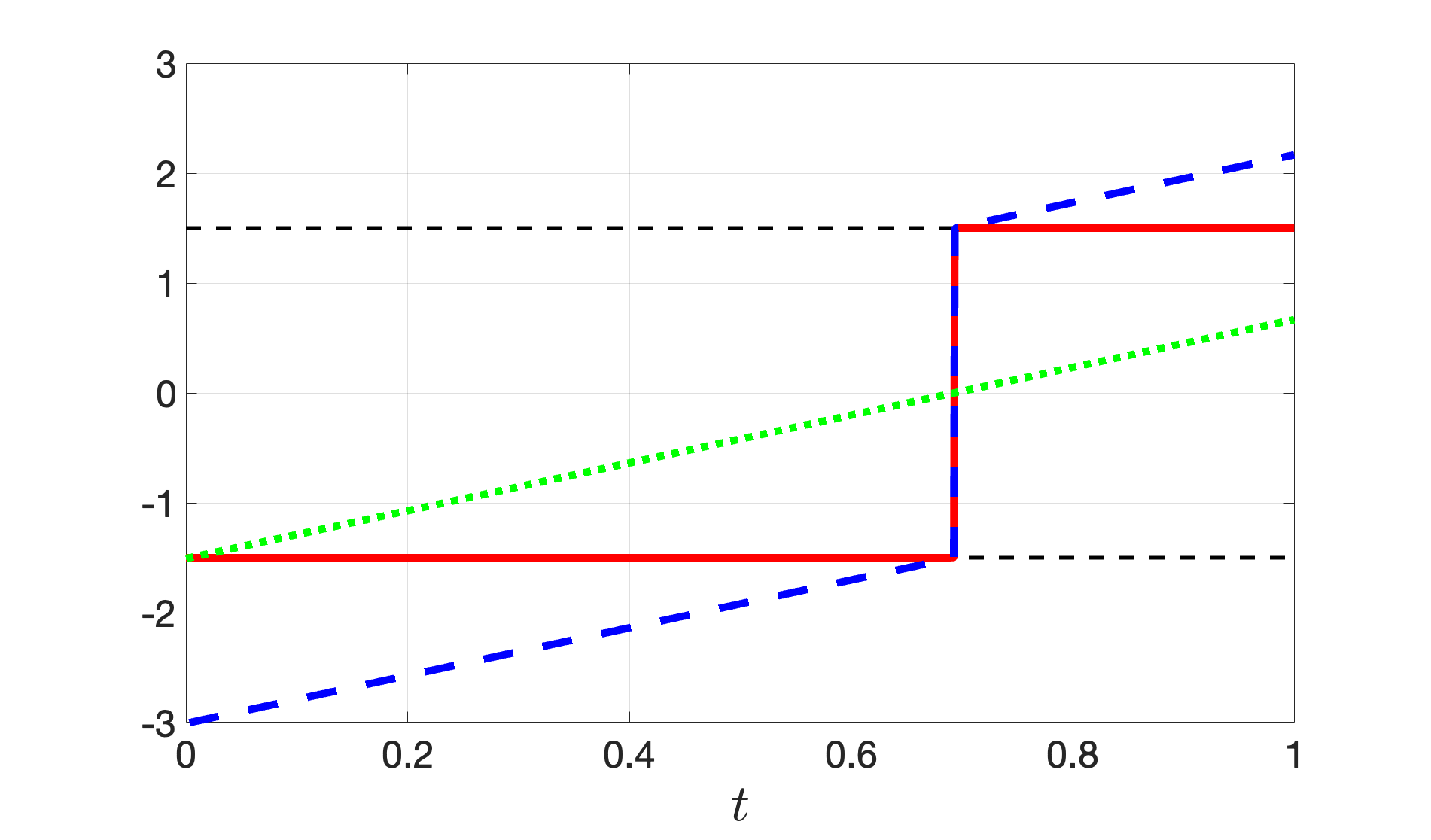}
\caption{$a = 1.5$}
\label{fig:a=1.5}
\end{subfigure}
\hfill
\begin{subfigure}{.49\textwidth}
\centering
\includegraphics[width=\textwidth]{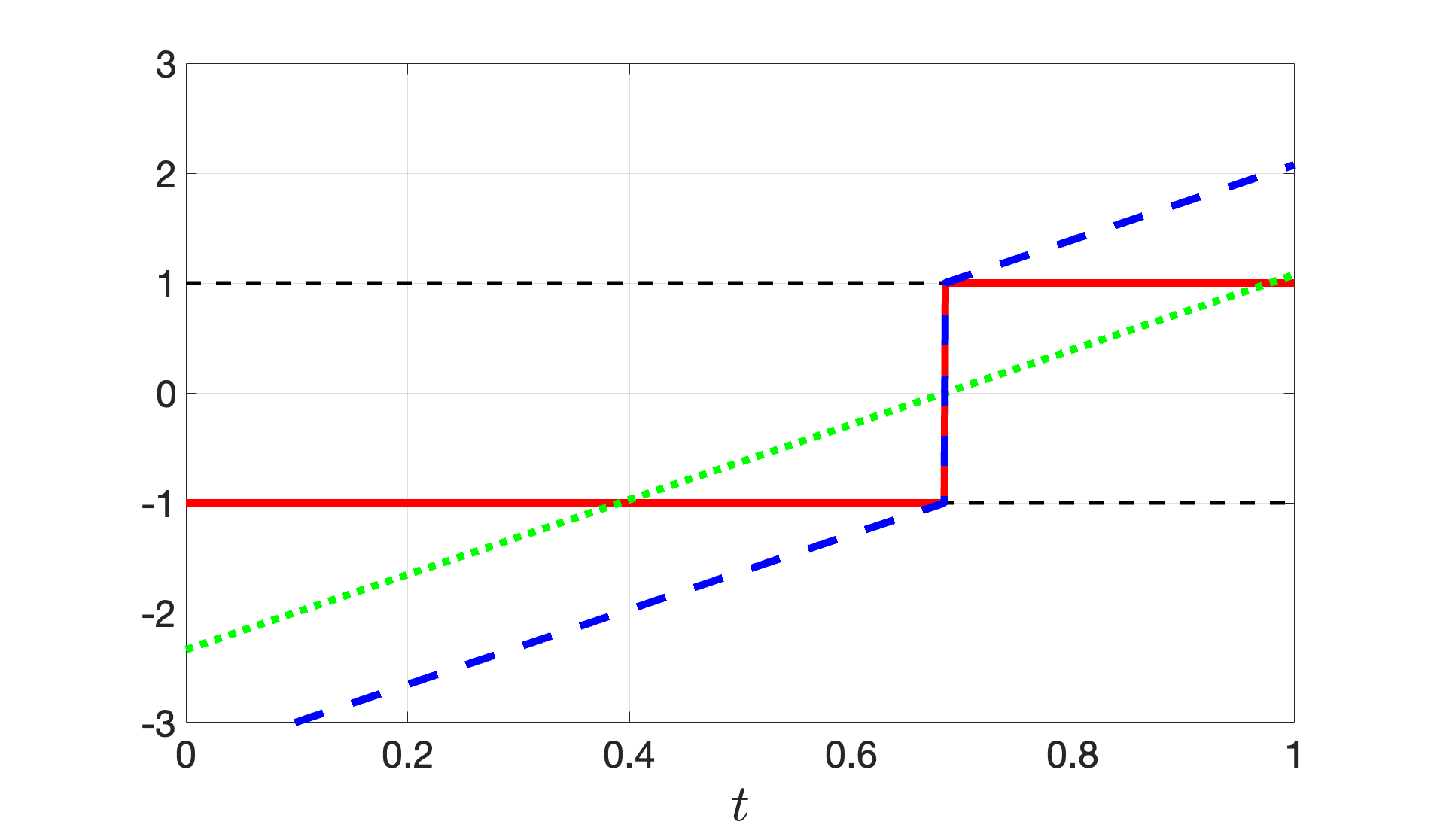}
\caption{$a = 1$}
\label{fig:a=1}
\end{subfigure}
\caption{\small\sf The best approximation control and the gap function for the double integrator for various values of $a$~\cite{BurKayMou2024}.}
\label{fig:DI}
\end{figure}

\ \\[-3mm] 
\noindent
{\bf Newton and Generalized Newton Methods to Solve (\ref{eqn1})--(\ref{eqn2}).}
As pointed out in the case study above, an immediate approach to solve the nonlinear system of equations \eqref{eqn1}--\eqref{eqn2} for $t_s$ and $c_1$ is to use the Newton method, after setting $r = a$ or $r = -a$ in~\eqref{eqn1}--\eqref{eqn2}.

It is well known that for convergence with Newton one needs to provide an initial guess that is ``close enough'' to the solution.  Graphical {\em portraits} displaying colour-coded number of iterations to {\em reach} a solution are useful tools in judging how robust a method is in terms of the choice of the initial iterates. We can perhaps rephrase this concern by asking the following question.  How large is the region in which, when an initial iterate (or guess) is chosen, one can get to a solution in a ``reasonable'' number of iterations?

This question also prompts one to consider variants of the Newton method that may perform better in a larger domain.  One such variant is the multivariate generalized Newton method proposed in~\cite{BurCalKay2021}, which itself is a generalization of the univariate version in~\cite{BurKaySab2012}.  While the Newton method solves $f(y) = {\bf 0}$, the generalized Newton method is the Newton method that solves $(s \circ f)(y) = {\bf 0}$, where the auxiliary function $s$ is differentiable and has a differentiable inverse so that the solutions of the equation remain the same, and $s$ is chosen in such a way that the Newton method applied to solving $(s \circ f)(y) = {\bf 0}$ exhibits more desirable convergence properties, notably a larger domain of convergence.  Obviously, when $s$ is the identity function the generalized Newton method reduces to the Newton method.  More recently a Levenberg--Marquardt-type extension of~\cite{BurCalKay2021} has also been proposed in~\cite{BurKay2025}, which, with carefully chosen parameters, potentially enlarges the region further in which convergence is achieved.

In Figures~\ref{fig:DI_Newton_basin}--\ref{fig:DI_gNewton_basin}, the portraits of the colour-coded number of iterations required to reach a solution of~\eqref{eqn1}--\eqref{eqn2} are shown for Newton and generalized Newton methods for boundary values $s_0 = s_f = v_f = 0$ and $v_0 = 1$, and infeasible values $a = 2$, $1.5$, $1$ and $0.5$.  The solutions obtained are correct to $15$ decimal places.  The portraits show that, with the Newton method, a solution is reached in a relatively large $t_s$$c_1$-plane region in less than $10$ iterations.  This region becomes larger with the generalized Newton method with $s(y) = (y_1^3, y_2)$.

\begin{figure}[t!]
\centering
\begin{subfigure}{.5\textwidth}
\centering
\includegraphics[width=\textwidth]{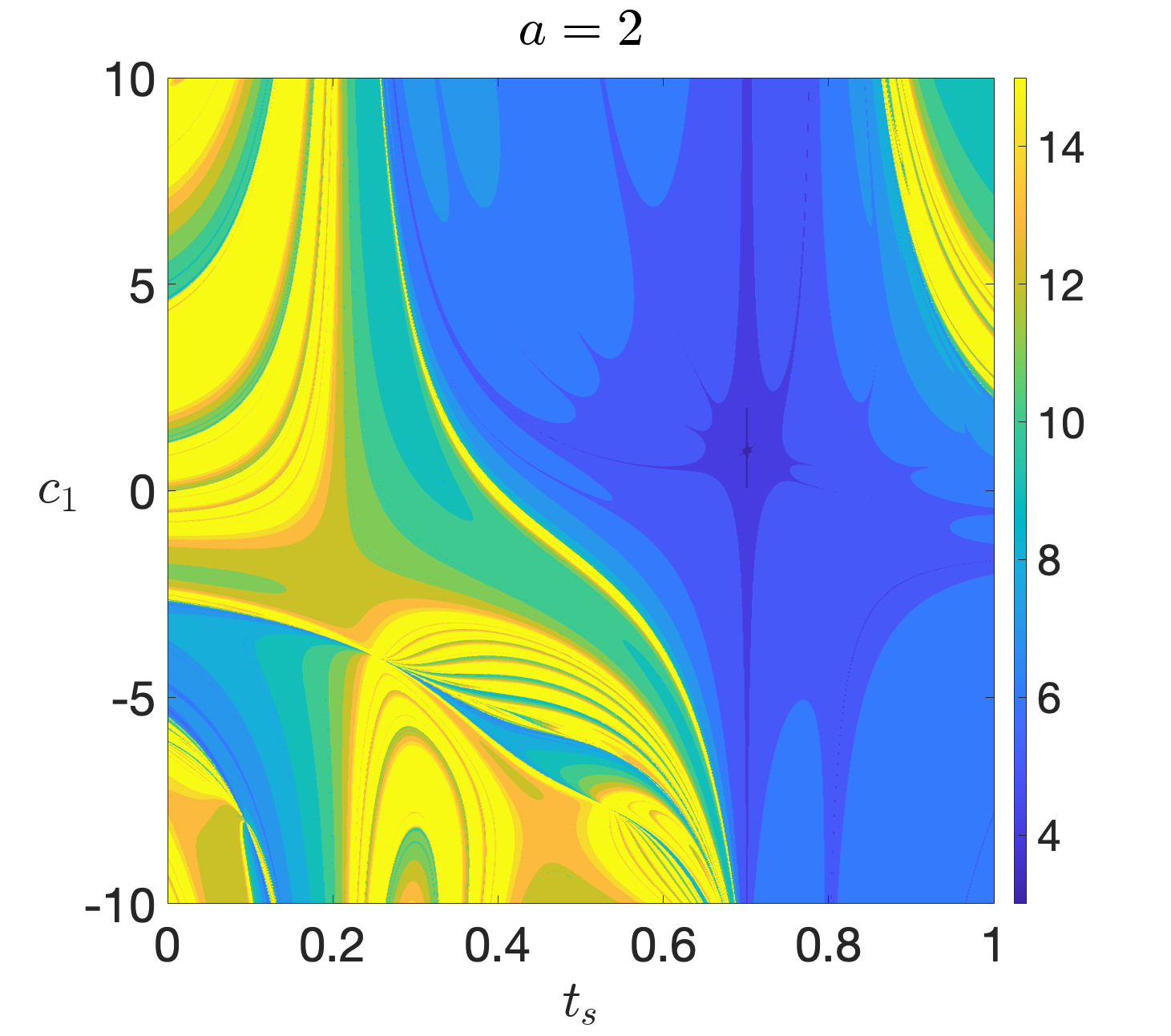}
\label{fig2:a=2}
\end{subfigure}
\hfill
\begin{subfigure}{.49\textwidth}
\centering
\includegraphics[width=\textwidth]{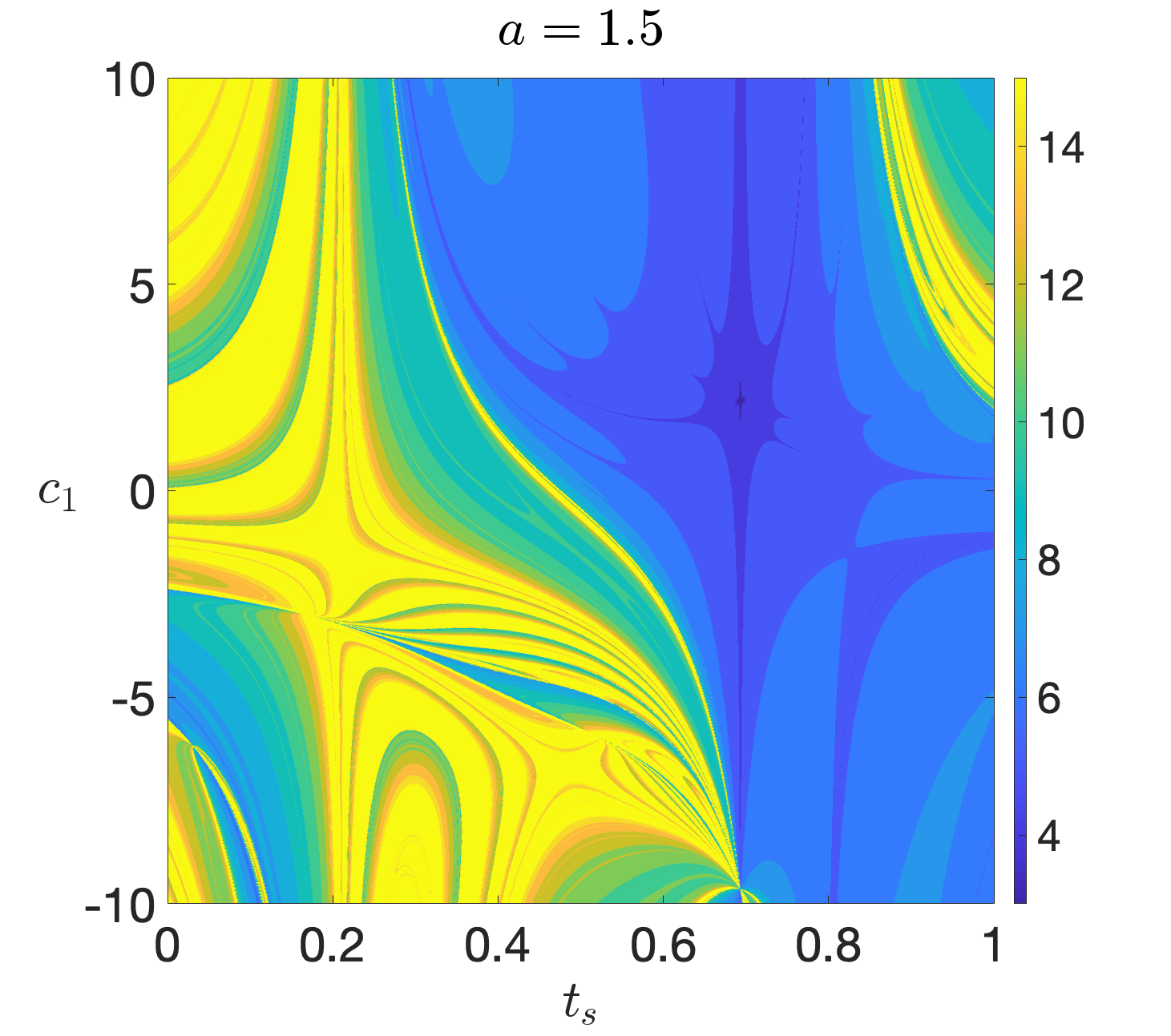}
\label{fig2:a=1.5}
\end{subfigure}
\begin{subfigure}{.5\textwidth}
\centering
\includegraphics[width=\textwidth]{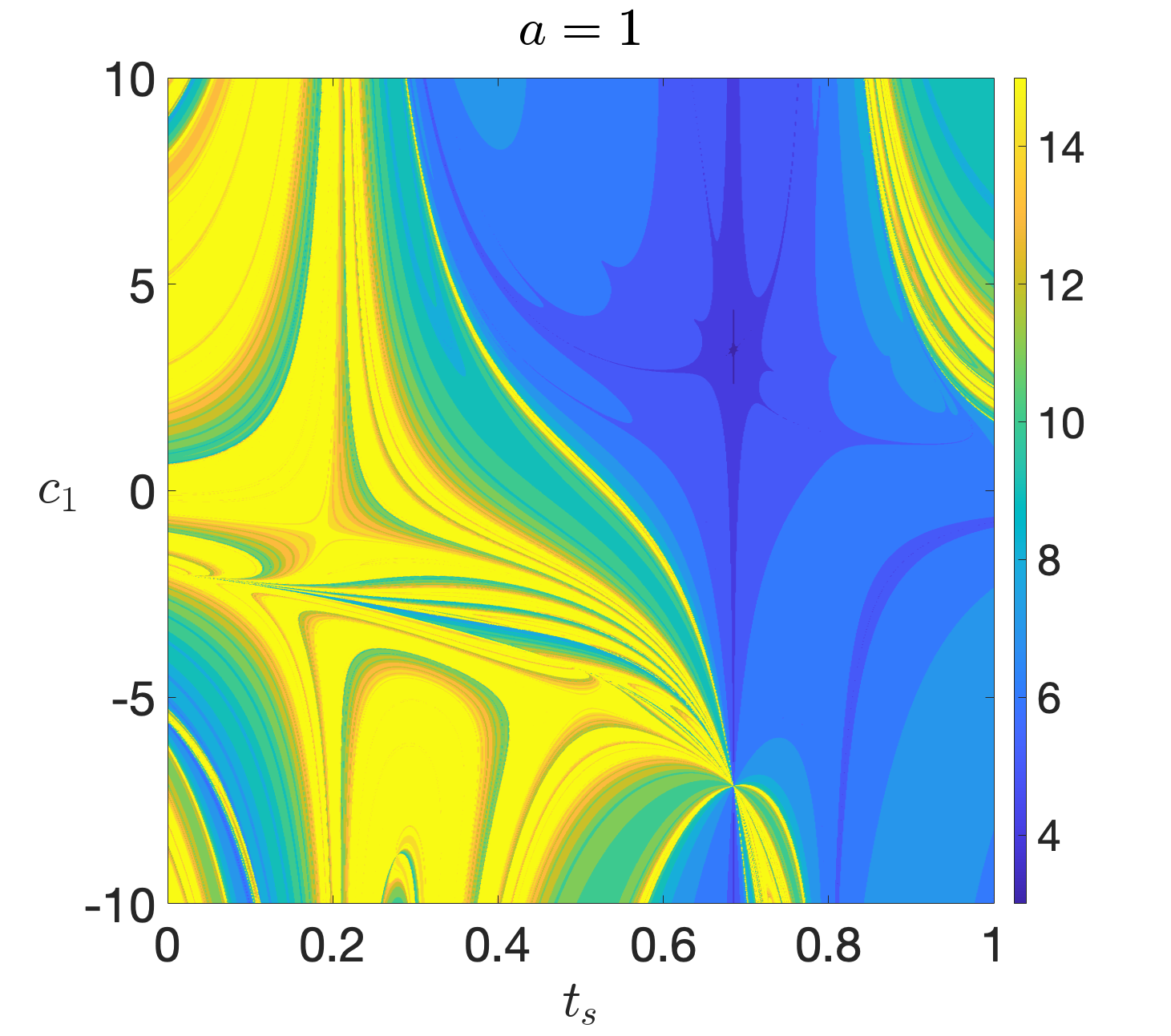}
\label{fig2:a=1}
\end{subfigure}
\hfill
\begin{subfigure}{.49\textwidth}
\centering
\includegraphics[width=\textwidth]{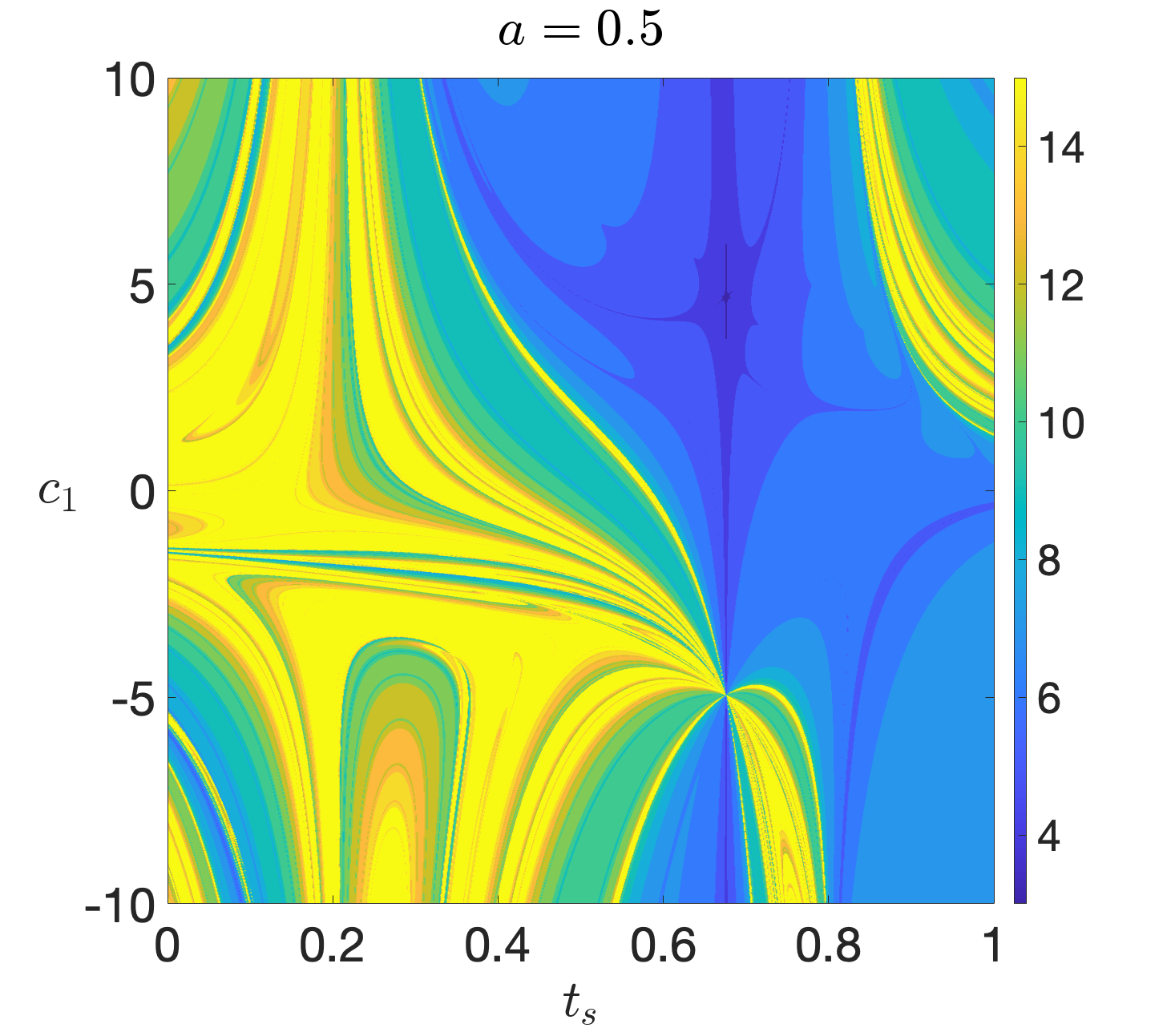}
\label{fig2:a=0.5}
\end{subfigure}
\caption{\small\sf Portraits of the colour-coded number of iterations required to reach a solution of~\eqref{eqn1}--\eqref{eqn2} via the Newton method.}
\label{fig:DI_Newton_basin}
\end{figure}
 
\begin{figure}[t!]
\centering
\begin{subfigure}{.5\textwidth}
\centering
\includegraphics[width=\textwidth]{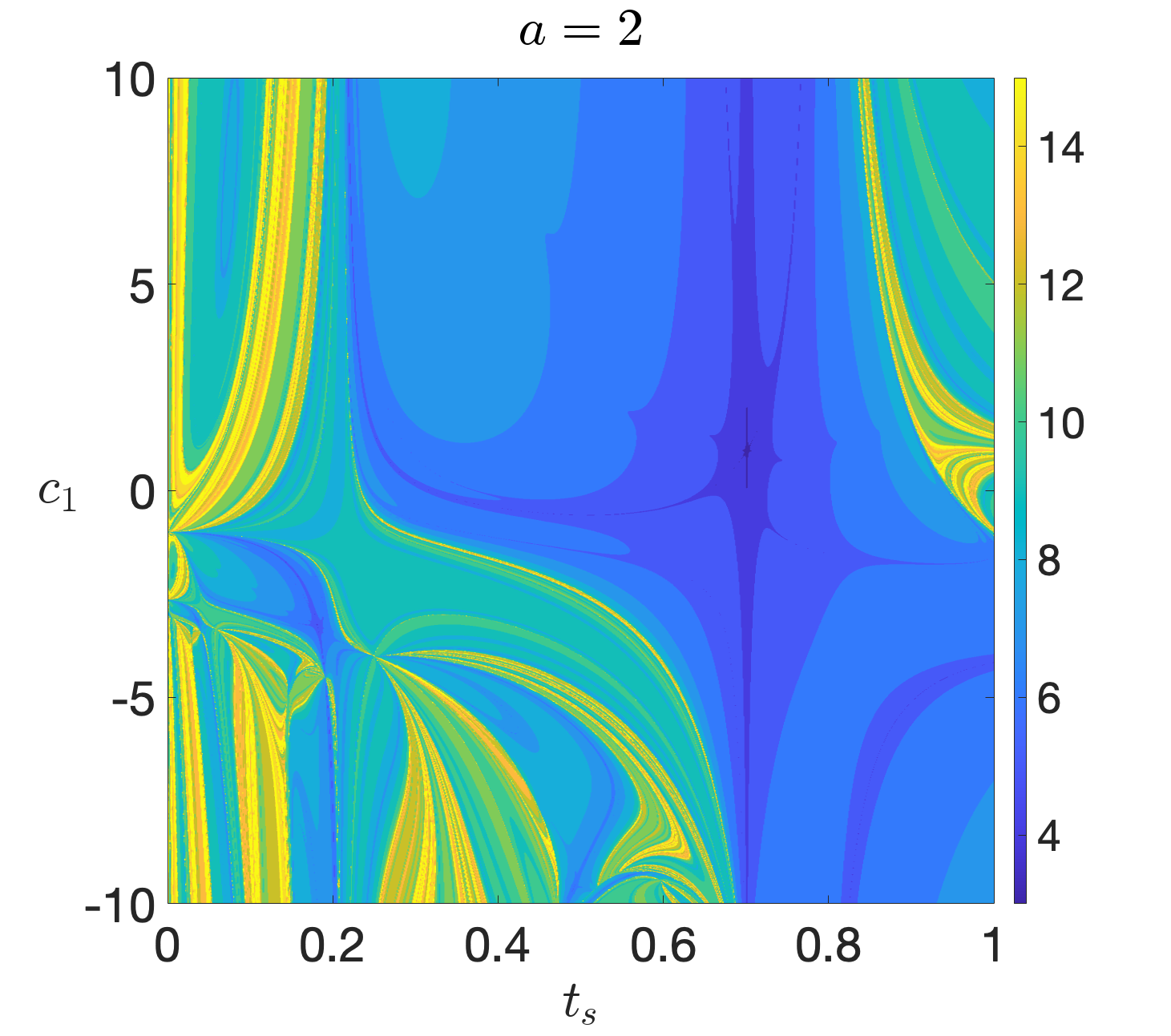}
\label{fig3:a=2}
\end{subfigure}
\hfill
\begin{subfigure}{.49\textwidth}
\centering
\includegraphics[width=\textwidth]{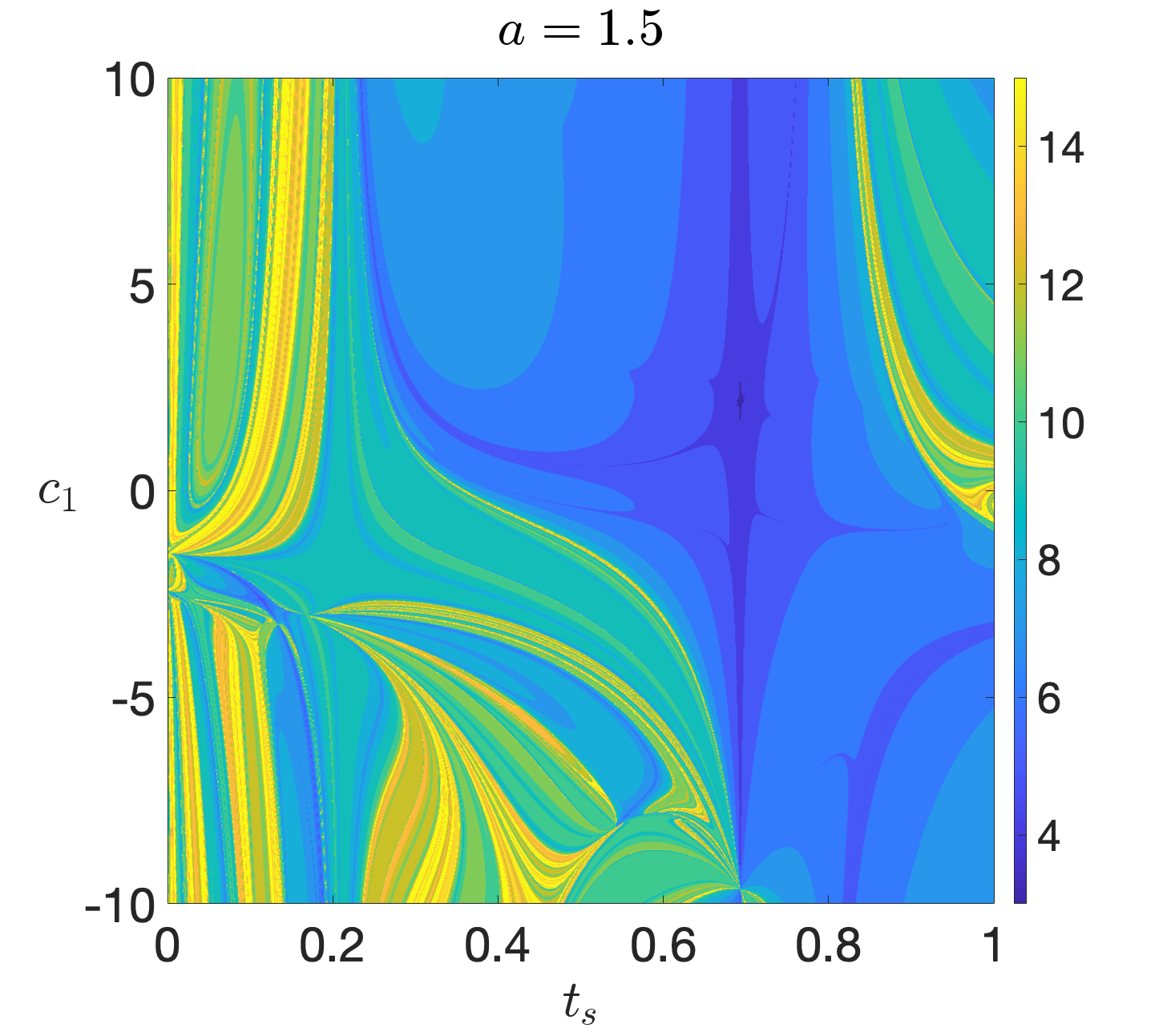}
\label{fig3:a=1.5}
\end{subfigure}
\begin{subfigure}{.5\textwidth}
\centering
\includegraphics[width=\textwidth]{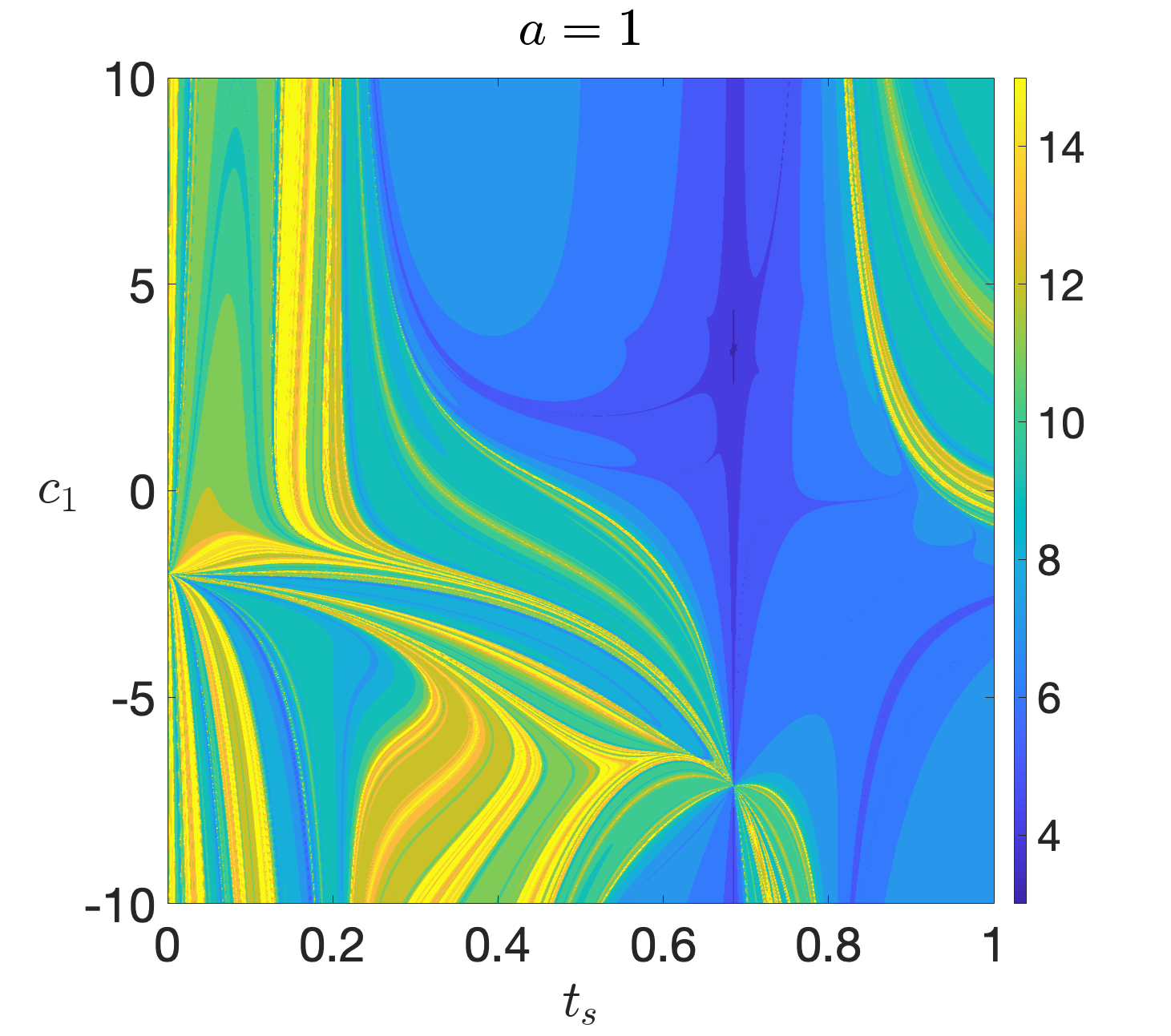}
\label{fig3:a=1}
\end{subfigure}
\hfill
\begin{subfigure}{.49\textwidth}
\centering
\includegraphics[width=\textwidth]{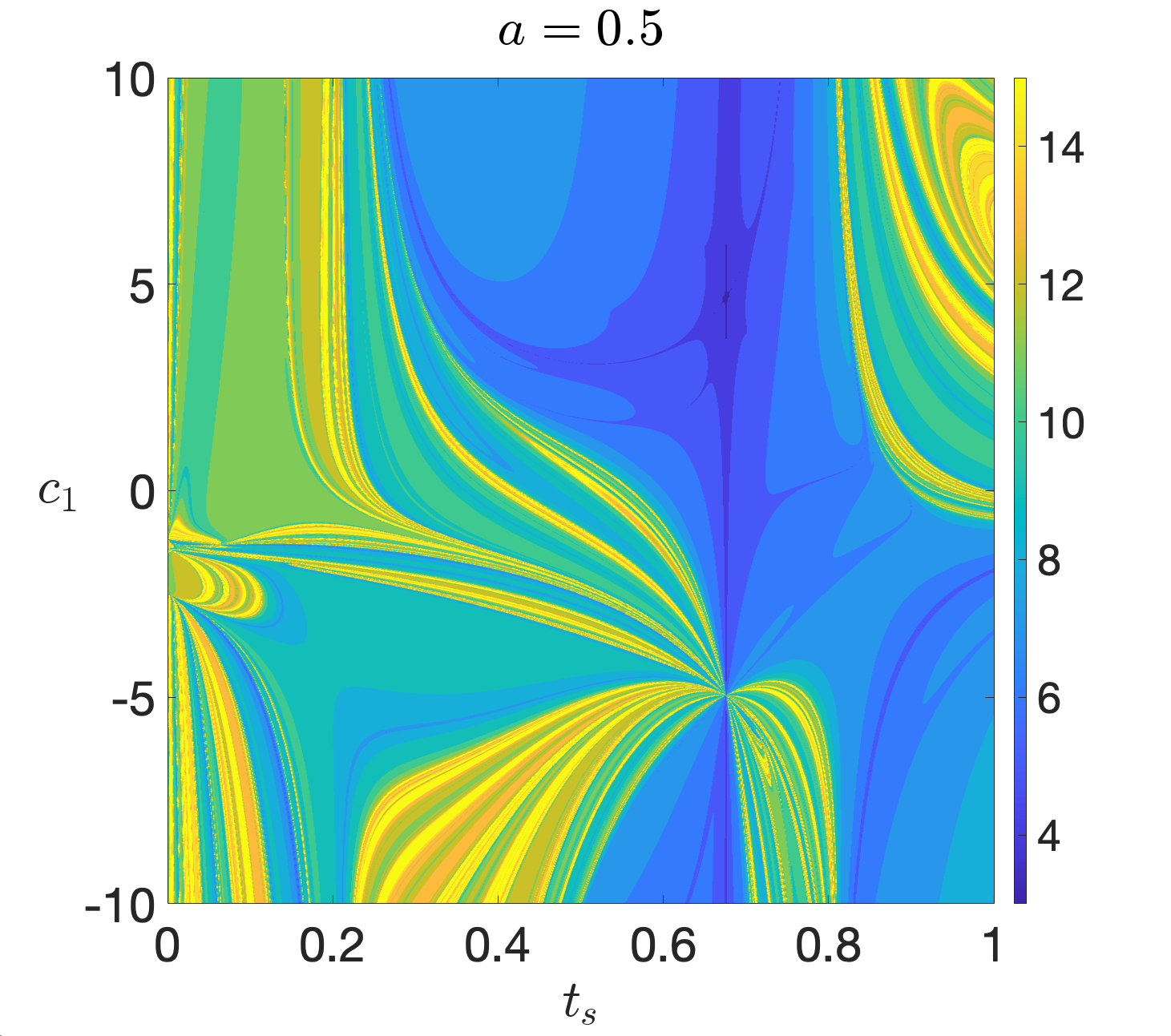}
\label{fig3:a=0.5}
\end{subfigure}
\caption{\small\sf Portraits of the colour-coded number of iterations required to reach a solution of~\eqref{eqn1}--\eqref{eqn2} via the generalized Newton method with $s(y) = (y_1^3, y_2)$.}
\label{fig:DI_gNewton_basin}
\end{figure}

\section{Douglas--Rachford Algorithm}
\label{sec:DR}

To define the Douglas--Rachford (DR) algorithm we require the \emph{proximal operators} of $f$ and $g$ where the proximal operator, or proximal mapping, of a function $h$ is defined by \cite[Definition~12.23]{BC2017}:
\begin{equation}\label{def:prox}
    {\rm Prox}_h(u) := \argmin_{y\in {\cal{L}}^2([t_0,t_f];\dR^m)}
    \left(h(y) + \frac{1}{2}\|y - u\|^2_{{\cal L}^2}  \right),
\end{equation}
for any $u\in {\cal L}^2([t_0,t_f];\dR^m)$. It is known that $\prox_f$ is firmly non-expansive \cite[Proposition~12.28]{BC2017}. Moreover, if $h = \iota_C$ then $\prox_{h} = P_{C}$ where $P_C$ is the \textit{projection operator} onto~$C$. Recall that the {\em Douglas--Rachford operator} is given as
\begin{align*}
    T &= (1-\lambda)\Id+\lambda(\Rr_{\tau g}\Rr_{\tau f}), \\[2mm]
    &= \Id+2\lambda(\prox_{\tau g}(2\prox_{\tau f}-\Id)-\prox_{\tau f}),
\end{align*}
where $\lambda\in\;(0,1)$. When $\lambda=1/2$ the above operator reduces to the 
DR operator introduced in \cite{LionsMer79}. 
See also \cite[Theorem~26.11]{BC2017} for a more general formula which allows for a variable $\lambda$.
Consider the following functions,
\begin{equation}\label{def:fg}
    f=\iota_{\cal B}+\frac{1}{2}\|\cdot\|^2 \quad\hbox{ and }\quad g=\iota_{\cal A}.
\end{equation}
Let $\tau>0$ and $\gamma:={{\frac{1}{\tau+1}}}\in(0,1)$. Using \cite[Proposition~24.8(i)]{BC2017},
\begin{align}
\prox_{\tau f}(u) & = \prox_{\tau\iota_{\cal B}}(\gamma u) = P_{\cal B}(\gamma u).
    \label{eqn:proxf}\\[2mm]
\prox_{\tau g}(u) &
    = \prox_{\tau\iota_{\cal A}}(u) = P_{\cal A}(u), \label{eqn:proxg} 
   \end{align}
Using \eqref{eqn:proxf} and \eqref{eqn:proxg},
\[
Tu = u + 2\lambda(P_{\cal A}(2P_{\cal B}(\gamma u)-u)-P_{\cal B}(\gamma u)).
\]

\begin{remark}\rm \label{rem:fg}
The convergence result  we state next holds in an arbitrary Hilbert space $U$ and is a straightforward consequence of \cite[Theorem~5.3]{incon2024}. Indeed, consider the functions $\{f,g\}$ defined in \eqref{def:fg}.  We note that $f$ is uniformly convex, $\dom f={\cal B}$ and $\dom g={\cal A}$. We also note that the gap vector $v=u_{\cal A}-u_{\cal B}=P_{\overline{{\rm ran}(\Id-T)}}(0)$.  Indeed, thanks to \cite[Lemma~5.1]{incon2024}, we have that 
\[
\overline{{\rm ran}(\Id-T)}=\overline{\dom f -\dom g}=\overline{{\cal B}-{\cal A}}.
\]
Hence, the vector of minimum norm in $\overline{{\rm ran}(\Id-T)}$ coincides with the vector of minimum norm in $\overline{{\cal B}-{\cal A}}$, and this is precisely the gap vector. With these observations, the following theorem is a direct corollary of \cite[Theorem~5.3]{incon2024}. 
\end{remark}    

\begin{theorem}[Convergence of the DR Algorithm]
\label{thm:DR:convergence:fg}
Let $\lambda\in (0,1)$ and set
\begin{equation}
    T =(1-\lambda) \Id +\lambda \Rr_g\Rr_f,
\end{equation}	
where $f$ and $g$ are as in \eqref{def:fg}. Let $v=u_{\cal A}-u_{\cal B}$ be the gap vector introduced in \eqref{gap_v} and define
\[
Z:=\argmin(-\scal{\cdot}{v}+f+g(\cdot-v)).
\]
Suppose that $Z\neq \varnothing$. Then $(\exists \overline{u}\in U )$
such that $Z=\{\overline{u}\}$.
Let $u\in U$. 
Then the following hold:
\begin{enumerate}	
    \item[(a)]
    \label{thm:DR:convergence:fg:i}	
    $\prox_fT^n u\to \overline{u}$.
    \item[(b)]
    \label{thm:DR:convergence:fg:ii}
    $\prox_g\Rr_fT^n u\to \overline{u}-v$.
\end{enumerate}	
\end{theorem}
\begin{proof} 
Since the function $f$ is uniformly convex, all conclusions follow directly from \cite[Theorem~5.3]{incon2024}.
\end{proof}

\subsection{DR Algorithm for the Infeasible Problem~(PDI)} 

The formulas~\eqref{eqn:projA} and \eqref{eqn:projB} below, which can be found in~\cite[Propositions~1 and 2]{BauBurKay2019}, provide the projections onto ${\cal A}$ and ${\cal B}$.  These formulas are in turn going to be used in the DR algorithm (Algorithm~\ref{algo:DR}) in the following.

\noindent
\textbf{(Projection onto ${\cal A}$).}
The projection $P_{\cal A}$ of $u^-\in {\cal L}^2([0,1];\dR)$ onto the set ${\cal A}$ is given by

\begin{equation}\label{eqn:projA}
P_{\cal A}(u^-)(t) = u^-(t) + c_1\,t + c_2\,,
\end{equation}
for all $t\in[0,1]$, where
\begin{eqnarray*}
c_1 &:=& 12\left(s_0+v_0-s_f+\int_0^1 (1-\tau)u^-(\tau)d\tau\right)-6\left(v_0-v_f+\int_0^1 u^-(\tau)d\tau\right), \\[1mm]
c_2 &:=& -6\left(s_0+v_0-s_f+\int_0^1 (1-\tau)u^-(\tau)d\tau\right)+2\left(v_0-v_f+\int_0^1 u^-(\tau)d\tau\right).
\end{eqnarray*}

\noindent
\textbf{(Projection onto ${\cal B}$).}
The projection $P_{{\cal B}}$ of $u^-\in {\cal L}^2([t_0,t_f];\dR)$ onto the set ${\cal B}$ is given by
\begin{equation}  \label{eqn:projB}
P_{{\cal B}}(u^-)(t) = \left\{\begin{array}{rl}
a\,, &\ \ \mbox{if\ \ } u^-(t)\ge a\,, \\[2mm]
u^-(t)\,, &\ \ \mbox{if\ \ } -a\le u^-(t)\le a\,, \\[2mm]
-a\,, &\ \ \mbox{if\ \ } u^-(t)\le -a\,, \\[2mm]
\end{array} \right.
\end{equation}
for all $t\in[t_0,t_f]$.

We implemented the DR algorithm for the infeasible problem~(PDI) as follows.

\begin{algorithm}
\caption{Douglas--Rachford algorithm}
\label{algo:DR}
\begin{description}
\item[Step 1] ({\em Initialization}) Choose two parameters $\gamma\in(0,1)$, $\lambda\in(0,1)$ and the initial iterate $u^0$ arbitrarily. 
Choose a small parameter $\epsilon>0$, and set $k=0$. 
\item[Step 2] ({\em Projection onto ${\cal B}$})  Set $u^- = \gamma u^{k}$. 
Compute $\widetilde{u} = P_{{\cal B}}(u^-)$ by using \eqref{eqn:projB}. 
\item[Step 3] ({\em Projection onto ${\cal A}$}) Set $u^- := 2\widetilde{u}-u^k$. 
Compute $\widehat{u} = P_{{\cal A}}(u^-)$ by using \eqref{eqn:projA}.
\item[Step 4] ({\em Update}) Set $u^{k+1} := u^k + 2\lambda(\widehat{u} - \widetilde{u})$.
\item[Step 5] ({\em Stopping criterion}) If $\|u^{k+1} - u^k\|_{{\cal L}^\infty} \le \epsilon$, then return $\widetilde{u}$ and stop.  
Otherwise, set $k := k+1$ and go to Step 2.

\end{description}
\end{algorithm}

\subsection{Case study}

The numerical experiments we present next were conducted using {\sc Matlab} version R2025b~\cite{Matlab} on a computer with an Intel i5-12600KF CPU and 32 GB of DDR5 RAM. In the experiments, we fixed $\lambda=0.5$ which is the classical DR algorithm. In the numerical implementation we represent the iterates $u^{k}$, $\widehat{u}$ and $\widetilde{u}$ of the control variable $u$ (which are all functions) by $N$ discrete values over a regular partition of their domains $[0,1]$, so that computations can be carried out. To terminate the algorithm we require that at least $99.9\%$ of values in the control variable satisfy the stopping criteria in Algorithm~\ref{algo:DR}, as otherwise the stopping criterion in Step~5 of Algorithm~\ref{algo:DR} is not satisfied easily since the solution for $u$ is piecewise constant (i.e., the solution $u$ has a jump) and we are using the ${\cal L}^\infty$-norm in the stopping criterion. To obtain the switching time $t_s$ we used the discretized optimal control and the equation of the line connecting the $u=-a$ and $u=a$: we simply set $u=0$ in the equation of the line to retrieve $t_s$. We note that all CPU times were recorded over at least 1,000 runs and then averaged.

Figure~\ref{fig:params_PDI} demonstrates the number of iterations required for the DR algorithm to converge for different parameter values $\lambda$ and $\gamma$ as well as different choices of $a$. From Figure~\ref{fig:params3D} we see that the performance of the algorithm improves as $\lambda$ and $\gamma$ approach 1, although a larger $\gamma$ is more important than a larger $\lambda$ for faster convergence. Although this figure is only provided for $a=1.5$, graphs for other values of $a < a_c$ can be seen to be of similar appearance. In Figure~\ref{fig:params2D} we fix $\lambda=0.5$ which is the classical DR algorithm and observe that for $a=0.1,0.5,1,2$ the values for $\gamma$ that provide the best performance are those closest to 1. We also observe that as $a$ gets closer to $a_c$, the DR algorithm requires a much larger number of iterations.

In Table~\ref{tbl:DR} we show results for the DR algorithm applied to Problem~(PDI) with $\lambda=0.5$, $\gamma=0.95$ where $a=0.1,0.5,1,1.5,2$. The errors in $t_s$ were calculated using the values for $t_s$ recorded in Table~\ref{table:DI_ts} as the `exact' values for $t_s$. For a fixed grid size $N$, as $a$ increases we observe slightly larger errors in $t_s$, an increased number of iterations and increased CPU times. As we increase the order of $N$ we see a decrease in the error in $t_s$ by a similar order. These experiments align with observations we have made from previous works for the consistent case: the closer the problem is to critically feasible, the more difficult the problem is for the DR algorithm.

Table~\ref{tbl:Ipopt} provides results using AMPL--Ipopt analogous to those in Table~\ref{tbl:DR}. We have used the AMPL--Ipopt suite to solve the direct discretization of Problem~(Pf) as in~\cite{BurKayMou2024, Caldwell2024} (first-discretize-then-optimize approach). Compared with Table~\ref{tbl:DR} we see that in general Ipopt produced smaller errors in $t_s$ for $N=10^3,10^4$ while the DR algorithm was better for $N=10^5$. For all choices of $N$ the DR algorithm was consistently two orders of magnitude faster than Ipopt. Having said this, we re-iterate that when $a$ is very near $a_c$, since the CPU time for the DR algorithm becomes prohibitively large, the AMPL--Ipopt suite should be the choice of optimization software to use.

\begin{figure}[t!]
    \centering
    \begin{subfigure}{0.54\textwidth}
        \includegraphics[width=\textwidth]{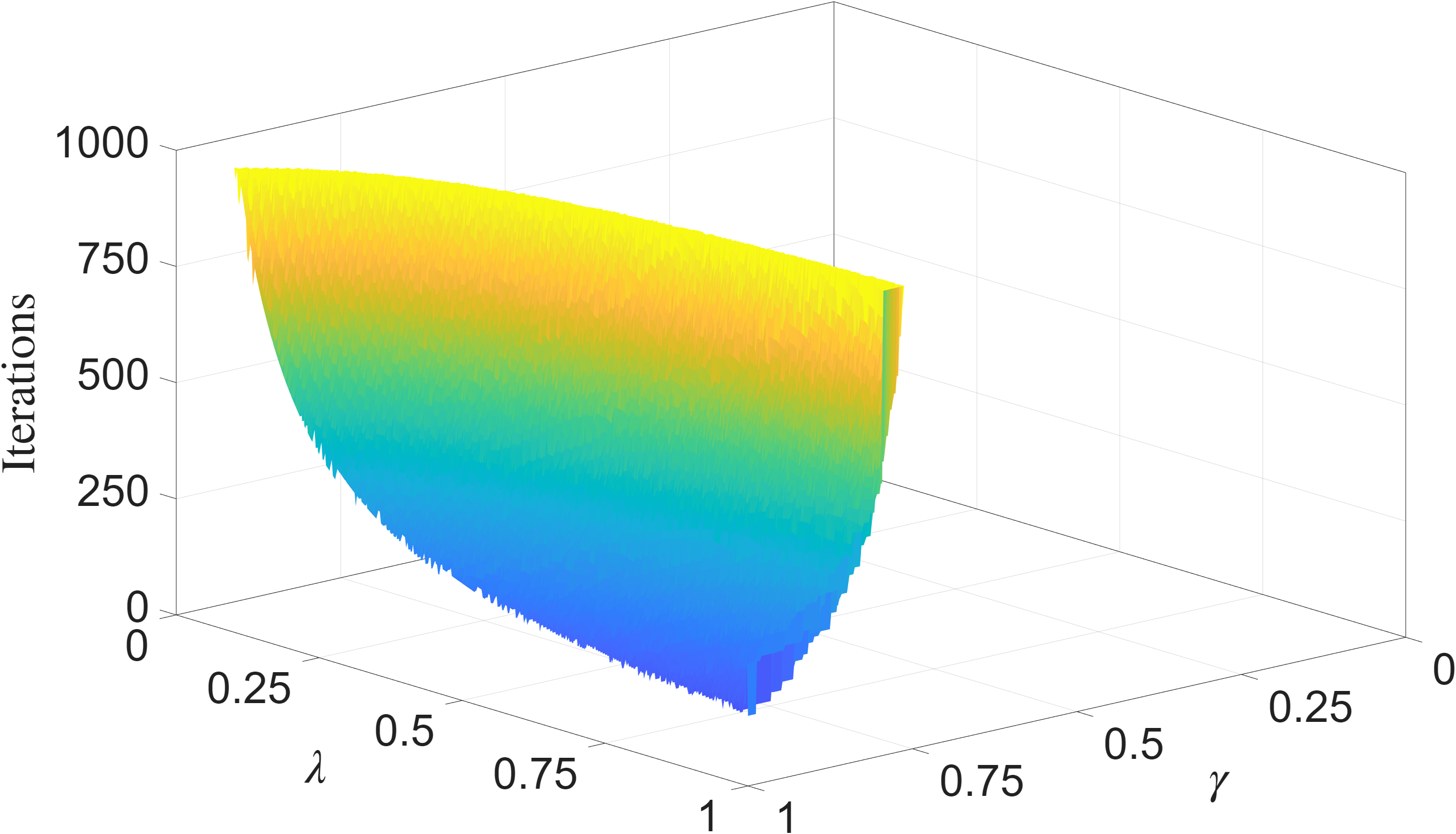}
        \caption{$a=1.5$}
        \label{fig:params3D}
    \end{subfigure}
    \begin{subfigure}{0.44\textwidth}
        \includegraphics[width=\textwidth]{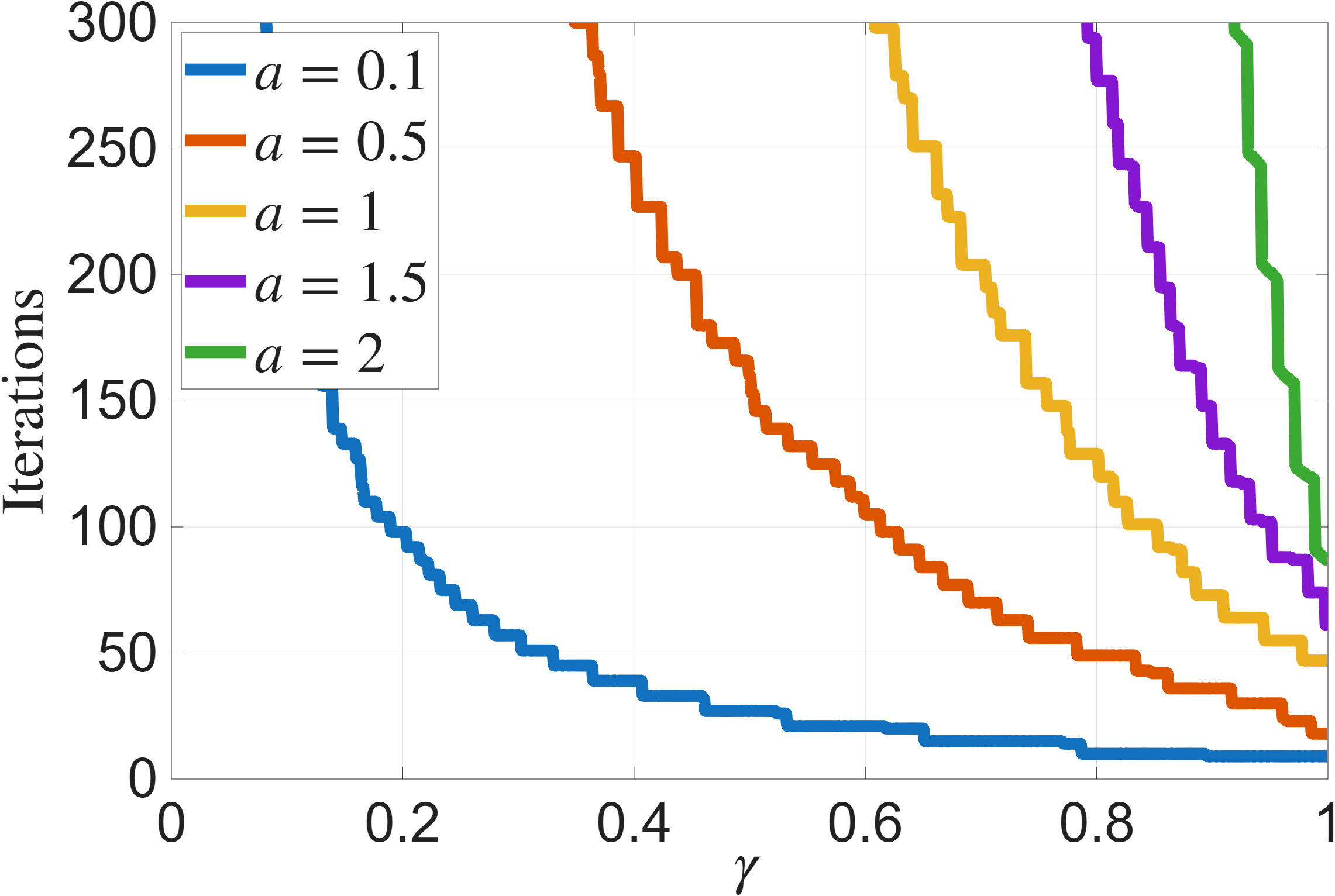}
        \caption{$\lambda=0.5$}
        \label{fig:params2D}
    \end{subfigure}
    \caption{Parameter curves for (PDI) using the DR algorithm with $\epsilon=10^{-6}$.}
    \label{fig:params_PDI}
\end{figure}

\begin{table}[t!]
    \centering
    \begin{tabular}{cc|ccccc}
        $N$ & & $a=0.1$ & $a=0.5$ & $a=1$ & $a=1.5$ & $a=2$ \\ \hline
        & Iterations            & $9$ & $30$ & $55$ & $102$ & $201$ \\
        \multirow{2}{*}{$10^3$} & $t_s$ & $0.6678$ & $0.6729$ & $0.6799$ & $0.6870$ & $0.6959$ \\
        & Error in $t_s$        & $7.6\times10^{-4}$ & $3.0\times10^{-3}$ & $4.9\times10^{-3}$ & $6.4\times10^{-3}$ & $5.3\times10^{-3}$ \\
        & CPU time [s]          & $2.5\times10^{-4}$ & $4.9\times10^{-4}$ & $8.1\times10^{-4}$ & $1.4\times10^{-3}$ & $2.6\times10^{-3}$ \\
        \hline
        & Iterations            & $11$ & $27$ & $50$ & $94$ & $238$ \\
        \multirow{2}{*}{$10^4$} & $t_s$ & $0.6684$ & $0.6757$ & $0.6844$ & $0.6928$ & $0.7005$ \\
        & Error in $t_s$        & $1.1\times10^{-4}$ & $1.5\times10^{-4}$ & $4.1\times10^{-4}$ & $5.7\times10^{-4}$ & $7.1\times10^{-4}$ \\
        & CPU time [s]          & $2.7\times10^{-3}$ & $4.6\times10^{-3}$ & $7.3\times10^{-3}$ & $1.3\times10^{-2}$ & $3.0\times10^{-2}$ \\
        \hline
        & Iterations            & $11$ & $27$ & $51$ & $94$ & $237$ \\
        \multirow{2}{*}{$10^5$} & $t_s$ & $0.6685$ & $0.6760$ & $0.6848$ & $0.6933$ & $0.7011$ \\
        & Error in $t_s$        & $1.5\times10^{-5}$ & $8.1\times10^{-5}$ & $4.3\times10^{-5}$ & $5.5\times10^{-5}$ & $7.1\times10^{-5}$ \\
        & CPU time [s]          & $2.6\times10^{-2}$ & $4.3\times10^{-2}$ & $6.8\times10^{-2}$ & $1.1\times10^{-1}$ & $2.6\times10^{-1}$
    \end{tabular}
    \caption{Switching times for (PDI) using the DR algorithm with $\lambda = 0.5$, $\gamma=0.95$ and $\epsilon=10^{-6}$.}
    \label{tbl:DR}
\end{table}

\begin{table}[h]
    \centering
    \begin{tabular}{cc|ccccc}
        $N$ & & $a=0.1$ & $a=0.5$ & $a=1$ & $a=1.5$ & $a=2$ \\ \hline
        \multirow{3}{*}{$10^3$} & $t_s$ & $0.6678$ & $0.6750$ & $0.6843$ & $0.6926$ & $0.7005$ \\
        & Error in $t_s$        & $7.0\times10^{-4}$ & $9.3\times10^{-4}$ & $5.2\times10^{-4}$ & $7.4\times10^{-4}$ & $6.9\times10^{-4}$ \\
        & CPU time [s]          & $9.0\times10^{-2}$ & $1.2\times10^{-1}$ & $1.3\times10^{-1}$ & $1.4\times10^{-1}$ & $1.5\times10^{-1}$ \\
        \hline
        \multirow{3}{*}{$10^4$} & $t_s$ & $0.6682$ & $0.6775$ & $0.6851$ & $0.6933$ & $0.7010$ \\
        & Error in $t_s$        & $3.3\times10^{-4}$ & $1.6\times10^{-3}$ & $2.8\times10^{-4}$ & $2.1\times10^{-5}$ & $1.6\times10^{-4}$ \\
        & CPU time [s]          & $7.1\times10^{-1}$ & $1.0\times10^{0}$ & $1.2\times10^{0}$ & $1.3\times10^{0}$ & $1.3\times10^{0}$ \\
        \hline
        \multirow{3}{*}{$10^5$} & $t_s$ & $0.6675$ & $0.7003$ & $0.6844$ & $0.6929$ & $0.7005$ \\
        & Error in $t_s$        & $1.1\times10^{-3}$ & $2.4\times10^{-2}$ & $4.8\times10^{-5}$ & $3.8\times10^{-4}$ & $6.4\times10^{-4}$ \\
        & CPU time [s]          & $1.9\times10^{1}$ & $2.6\times10^{1}$ & $3.3\times10^{1}$ & $3.4\times10^{1}$ & $3.2\times10^{1}$
    \end{tabular}
    \caption{Switching times for (PDI) using AMPL--Ipopt with $\mbox{tol} =10^{-6}$.}
    \label{tbl:Ipopt}
\end{table}

\section{Conclusion}
\label{sec:conclusion}

We have studied the problem of finding the best approximation  control for the infeasible double integrator.  We have implemented the DR algorithm for the infeasible double integrator and carried out numerical experiments to demonstrate the effectiveness of the DR algorithm.  For the values of $a$ not so close to $a_c$ (where $a_c$ corresponds to the critically feasible case), the numerical experiments prompt us to use the DR algorithm in the future for those infeasible optimal control problems for which an analytical solution or simplification cannot be obtained.  In the future, it would also be interesting to study other operator splitting and projection methods such as the Peaceman--Rachford algorithm (the case where $\lambda\to 1$ in the relaxed DR algorithm) and the Dykstra algorithm, for infeasible optimal control problems.

\section*{Acknowledgment}
The authors offer their warm thanks to the anonymous reviewer for their insightful comments and suggestions that improved the manuscript. The authors also thank the editor for handling their submission efficiently.

\end{document}